\documentclass[12pt,a4paper,oneside]{article}
\usepackage{styling}
\usepackage[section]{placeins}

\begin{document}

\title{Invariant forms on irreducible modules of simple algebraic groups}
\author{Mikko Korhonen\thanks{Section de mathématiques, École Polytechnique Fédérale de Lausanne, CH-1015 Lausanne, Switzerland} \thanks{Email address: mikko.korhonen@epfl.ch} \thanks{The author was supported by a grant from the Swiss National Science Foundation (grant number $200021 \_ 146223$).}}
\maketitle
\hrule


\begin{abstract}
Let $G$ be a simple linear algebraic group over an algebraically closed field $\fieldsymbol$ of characteristic $p \geq 0$ and let $V$ be an irreducible rational $G$-module with highest weight $\lambda$. When $V$ is self-dual, a basic question to ask is whether $V$ has a non-degenerate $G$-invariant alternating bilinear form or a non-degenerate $G$-invariant quadratic form. 

If $p \neq 2$, the answer is well known and easily described in terms of $\lambda$. In the case where $p = 2$, we know that if $V$ is self-dual, it always has a non-degenerate $G$-invariant alternating bilinear form. However, determining when $V$ has a non-degenerate $G$-invariant quadratic form is a classical problem that still remains open. We solve the problem in the case where $G$ is of classical type and $\lambda$ is a fundamental highest weight $\omega_i$, and in the case where $G$ is of type $A_l$ and $\lambda = \omega_r + \omega_s$ for $1 \leq r < s \leq l$. We also give a solution in some specific cases when $G$ is of exceptional type.

As an application of our results, we refine Seitz's $1987$ description of maximal subgroups of simple algebraic groups of classical type. One consequence of this is the following result. If $X < Y < \SL(V)$ are simple algebraic groups and $V \downarrow X$ is irreducible, then one of the following holds: (1) $V \downarrow Y$ is not self-dual; (2) both or neither of the modules $V \downarrow Y$ and $V \downarrow X$ have a non-degenerate invariant quadratic form; (3) $p = 2$, $X = \SO(V)$, and $Y = \Sp(V)$.
\end{abstract}

\maketitle

\newpage
\tableofcontents
\newpage

\section{Introduction}

\noindent Let $V$ be a finite-dimensional vector space over an algebraically closed field $\fieldsymbol$ of characteristic $p \geq 0$.

A fundamental problem in the study of simple linear algebraic groups over $\fieldsymbol$ is the determination of maximal closed connected subgroups of simple groups of classical type ($\SL(V)$, $\Sp(V)$ and $\SO(V)$). Seitz \cite{SeitzClassical} has shown that up to a known list of examples, these are given by the images of $p$-restricted, tensor-indecomposable irreducible rational representations $\varphi: G \rightarrow \GL(V)$ of simple algebraic groups $G$ over $\fieldsymbol$.

Then given such an irreducible representation $\varphi$, one should still determine which of the groups $\SL(V)$, $\Sp(V)$ and $\SO(V)$ contain $\varphi(G)$. In most cases the answer is known. 

\begin{itemize}
\item If $V$ is not self-dual, then $\varphi(G)$ is only contained in $\SL(V)$. Furthermore, we know when $V$ is self-dual (see Section \ref{formssection}).
\item If $p \neq 2$ and $V$ is self-dual, then $\varphi(G)$ is contained in $\Sp(V)$ or $\SO(V)$, but not both \cite[Lemma 78, Lemma 79]{SteinbergNotes}. Furthermore, we know for which irreducible representations the image is contained in $\Sp(V)$ and for which the image is contained in $\SO(V)$ (see Section \ref{formssection}).
\item If $p = 2$ and $V$ is self-dual, then $\varphi(G)$ is contained in $\Sp(V)$ \cite[Lemma 1]{Fong}.
\end{itemize}

Currently what is still missing is a method for determining in characteristic two when exactly $\varphi(G)$ is contained in $\SO(V)$. This problem is the main subject of this paper, and we can state it equivalently as follows.

\begin{prob}\label{mainproblem}
Assume that $p = 2$ and let $L_G(\lambda)$ be an irreducible $G$-module with highest weight $\lambda$. When does $L_G(\lambda)$ have a non-degenerate $G$-invariant quadratic form?
\end{prob}

This is a nontrivial open problem. There is some literature on the subject \cite{Willems}, \cite{SinWillems}, \cite{GowWillems1}, \cite{GowWillems2}, \cite{GaribaldiNakano}, but currently only partial results are known. The main result of this paper is a solution to Problem \ref{mainproblem} in the following cases:

\begin{itemize}
\item $G$ is of classical type ($A_l$, $B_l$, $C_l$ or $D_l$) and $\lambda$ is a fundamental dominant weight $\omega_r$ for some $1 \leq r \leq l$ (Theorem \ref{mainprop}).
\item $G$ is of type $A_l$ and $\lambda = \omega_r + \omega_s$ for $1 \leq r < s \leq l$ (Theorem \ref{quadpropA}).
\end{itemize}


In the case where $G$ is of exceptional type, we will give some partial results in Section \ref{exceptionalsection}. For $G$ of type $G_2$ and $F_4$, we are able to give a complete solution (Proposition \ref{prop:g2}, Proposition \ref{F4result}). For types $E_6$, $E_7$, and $E_8$, we give the answer for some specific $\lambda$ (Table \ref{table:typeE}). In the final section of this paper, we will give various applications of our results and describe some open problems motivated by Problem \ref{mainproblem}. 

One particular application, given in subsection \ref{subsec:applmax}, is a refinement of Seitz's \cite{SeitzClassical} description of maximal subgroups of simple algebraic groups of classical type. In \cite{SeitzClassical}, Seitz gives a full list of all non-maximal irreducible subgroups of $\SL(V)$, but the question of which classical groups contain the image of an irreducible representation is not considered. For example, it is possible that we have a proper inclusion $X < Y$ of irreducible subgroups of $\SL(V)$ such that $X$ is a maximal subgroup of $\SO(V)$. In subsection \ref{subsec:applmax}, we go through the list given by Seitz and describe when exactly such inclusions occur. In particular, our results have the consequence (Theorem \ref{theorem:maxconseq}) that if $X < Y < \SL(V)$ are simple algebraic groups and $V \downarrow X$ is irreducible, then one of the following holds:

\begin{enumerate}[(i)]
\item The module $V \downarrow Y$ is not self-dual;
\item Both $V \downarrow X$ and $V \downarrow Y$ have an invariant quadratic form;
\item Neither of $V \downarrow X$ or $V \downarrow Y$ has an invariant quadratic form;
\item $p = 2$, $X = \SO(V)$ and $Y = \Sp(V)$.
\end{enumerate}

The general approach for the proofs of our main results is as follows. A basic method used throughout is Theorem 9.5. from \cite{GaribaldiNakano} (recorded here in Proposition \ref{char2quad}), which allows one to determine whether $L_G(\lambda)$ is orthogonal (when $p = 2$) by computing within the Weyl module $V_G(\lambda)$. For $G$ of classical type and $V$ irreducible with fundamental highest weight, we will first prove our result in the case where $G$ is of type $C_l$ (Proposition \ref{fundamentalquadC}). From this the result for other classical types is a fairly straightforward consequence (Theorem \ref{mainprop}). 

In the case where $G$ is of type $C_l$ and $\lambda = \omega_r$, and in the case where $G$ is of type $A_l$ and $\lambda = \omega_r + \omega_s$, the proofs of our results are heavily based on various results from the literature on the representation theory of $G$. We will use results about the submodule structure of the Weyl module $V_G(\lambda)$ found in \cite{PremetSuprunenko}, \cite{Adamovich1} and \cite{Adamovich2}. We will also need the first cohomology groups of $L_G(\lambda)$ which were computed in \cite{KleshchevSheth} and \cite[Corollary 3.6]{KleshchevSheth2}. One more key ingredient in our proof will be the results of Baranov and Suprunenko in \cite{BaranovSuprunenkoC} and \cite{BaranovSuprunenkoA}, which give the structure of the restrictions of $L_G(\lambda)$ to certain subgroups defined in terms of the natural module of $G$.

\section*{Notation and terminology}

\noindent We fix the following notation and terminology. Throughout the whole text, let $\fieldsymbol$ be an algebraically closed field of characteristic $p \geq 0$. All groups that we consider are linear algebraic groups over $\fieldsymbol$, and by a subgroup we always mean a closed subgroup. All modules and representations will be finite-dimensional and rational. 

Unless otherwise mentioned, $G$ denotes a simply connected simple algebraic group over $\fieldsymbol$ with $l = \rank G$, and $V$ will be a finite-dimensional vector space over $\fieldsymbol$. Throughout we will view $G$ as its group of rational points over $K$, and most of the time $G$ will studied either as a Chevalley group constructed with the usual Chevalley construction (see e.g. \cite{SteinbergNotes}), or as a classical group with its natural module (i.e. $G = \SL(V)$, $G = \Sp(V)$ or $G = \SO(V)$). We will occasionally denote $G$ by its type, so notation such as $G = C_l$ means that $G$ is a simply connected simple algebraic group of type $C_l$. 

We fix the following notation, as in \cite{JantzenBook}.

\begin{itemize}
\item $T$: a maximal torus of $G$, with character group $X(T)$.
\item $X(T)^+$: the set of dominant weights for $G$, with respect to some system of positive roots.
\item $\ch V:$ the character of a $G$-module $V$. Here $\ch V$ is an element of $\Z[X(T)]$. 
\item $\omega_1, \omega_2, \ldots, \omega_l:$ the fundamental dominant weights in $X(T)^+$. We use the standard Bourbaki labeling of the simple roots, as given in \cite[11.4, pg. 58]{Humphreys}.
\item $L(\lambda)$, $L_G(\lambda):$ the irreducible $G$-module with highest weight $\lambda \in X(T)^+$.
\item $V(\lambda)$, $V_G(\lambda):$ the Weyl module for $G$ with highest weight $\lambda \in X(T)^+$.
\item $\rad V(\lambda):$ unique maximal submodule of $V(\lambda)$.
\end{itemize}

For a dominant weight $\lambda \in X(T)^+$, we can write $\lambda = \sum_{i = 1}^l m_i \omega_i$ where $m_i \in \Z_{\geq 0}$. We say that $\lambda$ is \emph{$p$-restricted} if $p = 0$, or if $p > 0$ and $0 \leq m_i \leq p-1$ for all $1 \leq i \leq l$. The irreducible representation $L_G(\lambda)$ is said to be \emph{$p$-restricted} if $\lambda$ is $p$-restricted.

A bilinear form $b$ is \emph{non-degenerate}, if its \emph{radical} $\rad b = \{v \in V: b(v,w) = 0 \text{ for all } w \in V \}$ is zero. For a quadratic form $Q : V \rightarrow \fieldsymbol$ on a vector space $V$, its \emph{polarization} is the bilinear form $b_Q$ defined by $b_Q(v,w) = Q(v+w) - Q(v) - Q(w)$ for all $v, w \in V$.  We say that $Q$ is \emph{non-degenerate}, if its \emph{radical} $\rad Q = \{v \in \rad b_Q : Q(v) = 0 \}$ is zero.

For a $\fieldsymbol G$-module $V$, a bilinear form $(-,-)$ is \emph{$G$-invariant} if $(gv,gw) = (v,w)$ for all $g \in G$ and $v,w \in V$. A quadratic form $Q: V \rightarrow \fieldsymbol$ is \emph{$G$-invariant} if $Q(gv) = Q(v)$ for all $g \in G$ and $v \in V$. We say that $V$ is \emph{symplectic} if it has a non-degenerate $G$-invariant alternating bilinear form, and we say that $V$ is \emph{orthogonal} if it has a non-degenerate $G$-invariant quadratic form.

Note that if $V$ has a $G$-invariant bilinear form, then for $\lambda, \mu \in X(T)$ the weight spaces $V_{\lambda}$ and $V_{\mu}$ are orthogonal if $\lambda \neq -\mu$. Thus to compute the form on $V$ it is enough to work in the zero weight space of $V$ and $V_{\lambda} \oplus V_{-\lambda}$ for nonzero $\lambda \in X(T)$. For a $G$-invariant quadratic form $Q$ on $V$, we have $Q(v) = 0$ for any weight vector $v \in V$ with non-zero weight.

Given a morphism $\phi: G' \rightarrow G$ of algebraic groups, we can twist representations of $G$ with $\phi$. That is, if $\rho: G \rightarrow \GL(V)$ is a representation of $G$, then $\rho \phi$ is a representation of $G'$. We denote the corresponding $G'$-module by $V^\phi$. When $p > 0$, we denote by $F: G \rightarrow G$ the Frobenius endomorphism induced by the field automorphism $x \mapsto x^p$ of $\fieldsymbol$, see for example \cite[Lemma 76]{SteinbergNotes}. When $G$ is simply connected and $\lambda \in X(T)^+$, we have $L_G(p \lambda) \cong L_G(\lambda)^F$.

If a representation $V$ of $G$ has composition series $V = V_1 \supset V_2 \supset \cdots \supset V_{t} \supset V_{t+1} = 0$ with composition factors $W_i \cong V_i / V_{i+1}$, we will occasionally denote this by $V = W_1 / W_2 / \cdots / W_t$.

\section*{Acknowledgements}

\noindent I am very grateful to Prof. Donna Testerman for suggesting the problem, and for her many helpful suggestions and comments on the earlier versions of this text. I would also like to thank Prof. Gary Seitz for providing the argument used in the proof of Lemma \ref{lemma:typeDrestriction}, and the two anonymous referees for their helpful comments for improvements.

\section{Invariant forms on irreducible $G$-modules}\label{formssection}


\noindent Let $L(\lambda)$ be an irreducible representation of a simple algebraic group $G$ with highest weight $\lambda = \sum_{i = 1}^l m_i \omega_i$. Write $d(\lambda) = \sum_{\alpha > 0} \langle \lambda, \alpha^\vee \rangle$, where the sum runs over the positive roots $\alpha$, where $\alpha^\vee$ is the coroot corresponding to $\alpha$, and $\langle \ ,\ \rangle$ is the usual dual pairing between $X(T)$ and the cocharacter group. 

We know that $L(\lambda)$ is self-dual if and only if $w_0(\lambda) = -\lambda$, where $w_0$ is the longest element in the Weyl group \cite[Lemma 78]{SteinbergNotes}. Furthermore, if $L(\lambda)$ is self-dual and $p \neq 2$, then $L(\lambda)$ is orthogonal if $d(\lambda)$ is even and symplectic if $d(\lambda)$ is odd \cite[Lemma 79]{SteinbergNotes}. Hence in characteristic $p \neq 2$ deciding whether an irreducible module is symplectic or orthogonal is a straightforward computation with roots and weights. In Table \ref{dualitytable}, we give the value of $d(\lambda) \mod{2}$ (when $\lambda = -w_0(\lambda)$) for each simple type, in terms of the coefficients $m_i$.

\begin{table}
\centering

\begin{tabular}{| l | l | l |}
\hline
Root system & When is $\lambda = -w_0(\lambda)$? & $d(\lambda) \mod{2}$ when $\lambda = -w_0(\lambda)$ \\
\hline
 & & \\
$A_l$ ($l \geq 1$) & iff $m_i = m_{l-i+1}$ for all $i$ & \begin{tabular}{@{}l@{}} $0$, when $l$ is even \\ $\frac{l+1}{2} \cdot m_{\frac{l+1}{2}}$, when $l$ is odd \end{tabular} \\ 
 & & \\ 
$B_l$ ($l \geq 2$) & always & \begin{tabular}{@{}l@{}} $0$, when $l \equiv 0, 3 \mod{4}$ \\ $m_l$, when $l \equiv 1, 2 \mod{4}$. \end{tabular} \\ 
 & & \\ 
$C_l$ ($l \geq 2$) & always & $m_1 + m_3 + m_5 + \cdots$ \\ 
 & & \\ 
$D_l$ ($l \geq 4$) & \begin{tabular}{@{}l@{}} $l$ even: always \\  $l$ odd: iff $m_l = m_{l-1}$ \end{tabular} & \begin{tabular}{@{}l@{}} $0$, when $l \not\equiv 2 \mod{4}$ \\ $m_l + m_{l-1}$, when $l \equiv 2 \mod{4}$. \end{tabular} \\ 
 & & \\ 
$G_2$ & always & $0$ \\ 
$F_4$ & always & $0$ \\ 
$E_6$ & iff $m_1 = m_6$ and $m_3 = m_5$ & $0$ \\ 
$E_7$ & always & $m_2 + m_5 + m_7$ \\ 
$E_8$ & always & $0$ \\ 
 & & \\
\hline
\end{tabular}

\caption{Values of $d(\lambda)$ modulo $2$ for a weight $\lambda = \sum_{i = 1}^l m_i \omega_i$}\label{dualitytable}

\end{table}

In characteristic $2$, it turns out that each nontrivial, irreducible self-dual module is symplectic, as shown by the following lemma found in \cite{Fong}. We include a proof for convenience.

\begin{lemma}\label{fonglemma}
Assume that $\Ch \fieldsymbol = 2$. Let $V$ be a nontrivial, irreducible self-dual representation of a group $G$. Then $V$ is symplectic for $G$.
\end{lemma}

\begin{proof}\cite{Fong} Since $V$ is self-dual, there exists an isomorphism $\varphi: V \rightarrow V^*$ of $G$-modules, which induces a non-degenerate $G$-invariant bilinear form $(-,-)$ defined by $(v,w) = \varphi(v)(w)$. Since $\varphi^t : V \rightarrow V^*$ defined by $\varphi^t(v)(w) = \varphi(w)(v)$ is also an isomorphism of $G$-modules, by Schur's lemma there exists a scalar $c$ such that $(v,w) = c(w,v)$ for all $v, w \in V$. Then $(v,w) = c^2 (v,w)$, so $c^2 = 1$ because $(-,-)$ is nonzero. Because we are in characteristic two, it follows that $c = 1$, so $(-,-)$ is a symmetric form. Now $\{v \in V : (v,v) = 0\}$ is a submodule of $G$. Because $V$ is nontrivial and irreducible, this submodule must be all of $V$ and so $(-,-)$ is alternating.\end{proof}

Lemma \ref{fonglemma} above shows that the image of any irreducible self-dual representation lies in $\Sp(V)$. The following general result reduces determining whether $L(\lambda)$ is orthogonal (in characteristic two) to a computation within the Weyl module $V(\lambda)$.

\begin{prop}\label{char2quad}
Assume that $\Ch \fieldsymbol = 2$. Let $\lambda \in X(T)^+$ be nonzero, $\lambda = -w_0(\lambda)$ and suppose that $\lambda \neq \omega_1$ if $G$ has type $C_l$. Then

\begin{enumerate}[\normalfont (i)]
\item The Weyl module $V(\lambda)$ has a nonzero $G$-invariant quadratic form $Q$, unique up to scalar.
\item The unique maximal submodule of $V(\lambda)$ is equal to $\rad b_Q$.
\item The irreducible module $L(\lambda)$ has a nonzero, $G$-invariant quadratic form if and only if $\rad Q = \rad b_Q$. If this is not the case, then $\rad Q$ is a submodule of $\rad b_Q$ with codimension $1$, and $\operatorname{H}^1(G, L(\lambda)) \neq 0$.
\item If $V(\lambda)$ has no trivial composition factor, then $L(\lambda)$ is orthogonal.
\end{enumerate}
\end{prop}

\begin{proof}
See Theorem 9.5. and Proposition 10.1. in \cite{GaribaldiNakano} for (i), (ii) and (iii). The claim in (iii) about $\operatorname{H}^1(G, L(\lambda))$ can also be deduced from \cite[Satz 2.5]{Willems}. The claim (iv) is a consequence of (iii), since $\operatorname{H}^1(G, L(\lambda)) \cong \Ext_G^1(\fieldsymbol, L(\lambda)) \cong \Hom_G(\rad V(\lambda), \fieldsymbol)$ by \cite[II.2.14]{JantzenBook}.
\end{proof}

In the case where $G$ is of type $C_l$ and $\lambda = \omega_1$, we have the following result which is well known. We include a proof for completeness.

\begin{prop}\label{naturalquad}
Assume that $\Ch \fieldsymbol = 2$ and that $G$ is of type $C_l$. Then $V = V(\omega_1) = L(\omega_1)$ has no nonzero $G$-invariant quadratic form.
\end{prop}

\begin{proof}
(\cite[Example 8.4]{GaribaldiNakano}) The claim follows from a more general result that any $G$-invariant rational map $f: V \rightarrow \fieldsymbol$ is constant. Indeed, for such $f$ we have $f(gv) = f(v)$ for all $g \in G$, $v \in V$. Because $G$ acts transitively on nonzero vectors in $V$, it follows that $f(w) = f(v)$ for all $w \in V - \{0 \}$. Thus $f(w) = f(v)$ for all $w \in V$ since $f$ is rational.
\end{proof}

\begin{lemma}\label{tensorlemma}
Let $V$ and $W$ be $G$-modules. If $V$ and $W$ are both symplectic for $G$, then $V \otimes W$ is orthogonal for $G$.
\end{lemma}

\begin{proof}
See \cite[Proposition 3.4]{SinWillems}, \cite[Proposition 9.2]{GaribaldiNakano}, or \cite[4.4, pg. 126-127]{KleidmanLiebeck}.
\end{proof}

\begin{remark}\label{restrictedremark}
Assume that $\Ch \fieldsymbol = 2$. Then lemmas \ref{fonglemma} and \ref{tensorlemma} show that if $V$ is a non-orthogonal irreducible $G$-module, then $V$ must be tensor indecomposable. By Steinberg's tensor product theorem, this implies that $V$ is a Frobenius twist of $L_G(\lambda)$ for some $2$-restricted weight $\lambda \in X(T)^+$. Therefore to determine which irreducible representations of $G$ are orthogonal, it suffices to consider $V = L_G(\lambda)$ with $\lambda \in X(T)^+$ a $2$-restricted dominant weight.
\end{remark}

\section{Fundamental representations for type $C_l$}\label{typeCsection}

\noindent Throughout this section, assume that $G$ is simply connected of type $C_l$, $l \geq 2$. In this section we determine when in characteristic $2$ a fundamental irreducible representation $L(\omega_r)$, $1 \leq r \leq l$, of $G$ has a nonzero $G$-invariant quadratic form. The answer is given by the following proposition, which we will prove in what follows.

\begin{prop}\label{fundamentalquadC}
Assume $\Ch \fieldsymbol = 2$. Let $1 \leq r \leq l$. Then $L(\omega_r)$ is not orthogonal if and only if $r = 1$, or $r = 2^{i+1}$ for some $i \geq 0$ and $l + 1 \equiv 2^{i+1} + 2^i + t \mod{2^{i+2}}$, where $0 \leq t < 2^i$.
\end{prop}

The following examples are immediate consequences of Proposition \ref{fundamentalquadC}.

\begin{esim}\label{exampleC2}
If $\Ch \fieldsymbol = 2$, then $L(\omega_2)$ is orthogonal if and only if $l \not\equiv 2 \mod{4}$.
\end{esim}

\begin{esim}\label{exampleC4}
If $\Ch \fieldsymbol = 2$, then $L(\omega_4)$ is orthogonal if and only if $l \not\equiv 5,6 \mod{8}$.
\end{esim}

\begin{esim}\label{spinexample}
If $\Ch \fieldsymbol = 2$, then $L(\omega_l)$ is orthogonal if and only if $l \geq 3$ (this was also proven in \cite[Corollary 4.3]{GowSpin}) and $L(\omega_{l-1})$ is orthogonal if and only if $l = 3$, $l = 4$ or $l \geq 6$.
\end{esim}

A rough outline for the proof of Proposition \ref{fundamentalquadC} is as follows. Various results from the literature about the representation theory of $G$ will reduce the claim to specific $r$ which must be considered. We will then study $V(\omega_r)$ by using a standard realization of it in the exterior algebra of the natural module $V$ of $G$. Here we can explicitly describe a nonzero $G$-invariant quadratic form $Q$ on $V(\omega_r)$. We will then find a vector $\gamma \in \rad V(\omega_r)$ such that $L(\omega_r)$ is orthogonal if and only if $Q(\gamma) = 0$. The proof is finished by computing $Q(\gamma)$.

\subsection{Representation theory}

\noindent The composition factors of $V(\omega_r)$ were determined in odd characteristic by Premet and Suprunenko in \cite[Theorem 2]{PremetSuprunenko}. Independently, the composition factors and the submodule structure of $V(\omega_r)$ were found in arbitrary characteristic by Adamovich in \cite{Adamovich1}, \cite{Adamovich2}. Using the results of Adamovich, it was shown in \cite[Corollary 2.9]{BaranovSuprunenkoC} that the result of Premet and Suprunenko also holds in characteristic two. 

To state the result about composition factors of $V(\omega_r)$, we need to make a few definitions first. Let $a, b \in \Z_{\geq 0}$ and write $a = \sum_{i \geq 0} a_i p^i$ and $b = \sum_{i \geq 0} b_i p^i$ for the expansions of $a$ and $b$ in base $p$. We say that \emph{$a$ contains $b$ to base $p$} if for all $i \geq 0$ we have $b_i = a_i$ or $b_i = 0$. \footnote{Note that in \cite{PremetSuprunenko} there is a typo, the definition on pg. 1313, line 9 should say ``for every $i = 0,1, \ldots, n$ $\ldots$''} For $r \geq 1$, we define $J_p(r)$ to be the set of integers $0 \leq j \leq r$ such that $j \equiv r \mod{2}$ and $l+1-j$ contains $\frac{r-j}{2}$ to base $p$. The main result of \cite{PremetSuprunenko}, also valid in characteristic $2$, can be then described as follows. Here we set $\omega_0 = 0$, so that $L(\omega_0)$ is the trivial irreducible module.

\begin{lause}\label{premetsuprunenko_thm}
Let $1 \leq r \leq l$. Then in the Weyl module $V(\omega_r)$, each composition factor has multiplicity $1$, and the set of composition factors is $\{L(\omega_j) : j \in J_p(r) \}$.
\end{lause}

In view of Proposition \ref{char2quad} (iii), it will also be useful to know when the first cohomology group $\operatorname{H}^1(G, L(\omega_r))$ is nonzero. This has been determined by Kleshchev and Sheth in \cite{KleshchevSheth} \cite[Corollary 3.6]{KleshchevSheth2}.

\begin{lause}
Let $1 \leq r \leq l$ and write $l+1-r = \sum_{i \geq 0} a_i p^i$ in base $p$. Then $\operatorname{H}^1(G, L(\omega_r)) \neq 0$ if and only if $r = 2(p - a_i)p^i$ for some $i$ such that $a_i > 0$, and either $a_{i+1} < p-1$ or $r < 2p^{i+1}$.
\end{lause}

In characteristic $2$, the result becomes the following.

\begin{seur}\label{extcorollary}
Assume that $\Ch \fieldsymbol = 2$. Let $1 \leq r \leq l$. Then $\operatorname{H}^1(G, L(\omega_r)) \neq 0$ if and only if $r = 2^{i+1}$ for some $i \geq 0$, and $l+1 \equiv 2^i + t \mod{2^{i+1}}$ for some $0 \leq t < 2^i$.
\end{seur}

Throughout this section we will consider subgroups $C_{l'} < C_l = G$, which are embedded into $G$ as follows. Consider $G = \Sp(V)$ and let $(-,-)$ be the non-degenerate $G$-invariant alternating form $(-,-)$ on $V$. Fix a symplectic basis $e_1, \ldots, e_l, e_{-1}, \ldots, e_{-l}$ of $V$, where $(e_i, e_{-i}) = 1 = -(e_{-i}, e_i)$ and $(e_i, e_j) = 0$ for $i \neq -j$. Then for $2 \leq l' < l$, the embedding $C_{l'} < C_l$ is $\Sp(V') < \Sp(V)$, where $V' \subseteq V$ has basis $e_{\pm 1}, \ldots, e_{\pm l'}$ and $\Sp(V')$ fixes the basis vectors $e_{\pm (l'+1)}, \ldots, e_{\pm l}$.

The module structure of the restrictions $L(\omega_r) \downarrow C_{l-1}$ have been determined by Baranov and Suprunenko in \cite[Theorem 1.1 (i)]{BaranovSuprunenkoC}. We will only need to know the composition factors which occur in such a restriction, and in this case the result is the following. Below we define $L_{C_{l-1}}(\omega_r) = 0$ for $r < 0$.

\begin{lause}\label{typeCrestriction}
Let $1 \leq r \leq l$ and assume that $l \geq 3$. Set $d = \nu_p(l + 1 - r)$, and $\varepsilon = 0$ if $l + 1 - r \equiv -p^{d} \mod{p^{d+1}}$ and $\varepsilon = 1$ otherwise. Then the character of $L_{C_{l}}(\omega_r) \downarrow C_{l-1}$ is given by $$\ch L_{C_{l-1}}(\omega_r) + 2 \ch L_{C_{l-1}}(\omega_{r-1}) + \left( \sum_{k = 0}^{d-1} 2 \ch L_{C_{l-1}}(\omega_{r - 2p^k}) \right) + \varepsilon \ch L_{C_{l-1}}(\omega_{r - 2p^d})$$ where the sum in the brackets is zero if $d = 0$. 
\end{lause}

Above $\nu_p$ denotes the $p$-adic valuation on $\Z$, so for $a \in \Z^+$ we have $\nu_p(a) = d$, where $d \geq 0$ is maximal such that $p^d$ divides $a$.  Note that if $d = \nu_2(l + 1 - r)$, then $l + 1 - r \equiv 2^d \equiv -2^d \mod{2^{d+1}}$. Therefore if $\Ch \fieldsymbol = 2$, we always have $\varepsilon = 0$ in Theorem \ref{typeCrestriction}. In particular, the composition factors occurring in $L(\omega_r) \downarrow C_{l-1}$ are $L_{C_{l-1}}(\omega_r)$ and $L_{C_{l-1}}(\omega_{r-2^k})$ for $0 \leq k \leq d$.

We will now give some applications of Theorem \ref{typeCrestriction} and Theorem \ref{premetsuprunenko_thm} in characteristic two, which will be needed in our proof of Proposition \ref{fundamentalquadC}.

\begin{lemma}\label{restrictionlemmaC}
Assume that $\Ch \fieldsymbol = 2$, and let $l \geq 2^{i+1}$, where $i \geq 0$. Suppose that $l + 1 \equiv 2^i + t \mod{2^{i+1}}$, where $0 \leq t < 2^i$. Then for $t+1 \leq j \leq 2^{i+1}$, the following hold:

\begin{enumerate}[\normalfont (i)]
\item All composition factors of the restriction $L(\omega_j) \downarrow C_{l-1}$ have the form $L(\omega_{j'})$ for some $l-1 \geq j' \geq t$.
\item $L_{C_l}(\omega_j) \downarrow C_{l-t}$ has no trivial composition factors.
\end{enumerate}
\end{lemma}

\begin{proof}
If $t = 0$ there is nothing to prove, so suppose that $t \geq 1$. It will be enough to prove (i) as then (ii) will follow by induction on $t$. Let $d = \nu_2(l+1-j)$. Suppose first that $0 \leq d < i+1$. Now $l+1-j \equiv t-j \mod{2^i}$, so then $\nu_2(l+1-j) = \nu_2(j-t)$. By Theorem \ref{typeCrestriction}, the composition factors occurring in $L(\omega_j) \downarrow C_{l-1}$ are $L(\omega_j)$ and $L(\omega_{j-2^{k}})$ for $0 \leq k \leq d$, so the claim follows since $\nu_2(j-t) = d$ and thus $j - 2^d \geq t$.

Consider then the case where $d \geq i+1$. Then $l+1-j \equiv 2^i + (t-j) \equiv 0 \mod{2^{i+1}}$, so $j-t \equiv 2^i \mod{2^{i+1}}$. On the other hand $0 < j-t < 2^{i+1}$, so $j-t = 2^i$. By Theorem \ref{typeCrestriction}, the composition factors occurring in $L(\omega_j) \downarrow C_{l-1}$ are $L(\omega_j)$ and $L(\omega_{j-2^k})$ for $0 \leq k \leq i$ (because $j-2^k < 0$ for $i+1 \leq k \leq d$), so again the claim follows.
\end{proof}

\begin{lemma}\label{firstarithmeticlemma}
Let $x \geq 2^{i+1}$, where $i \geq 0$. Suppose that $x \equiv 2^i \mod{2^{i+1}}$. If $0 \leq k \leq 2^i$ and $x-2k$ contains $2^i - k$ to base $2$, then $k = 0$ or $k = 2^i$. 
\end{lemma}

\begin{proof}
If $i = 0$ there is nothing to do, so suppose that $i > 0$. Replacing $k$ by $2^i - k$, we see that it is equivalent to prove that if $x+2k$ contains $k$ to base $2$, then $k = 0$ or $k = 2^i$. 

Suppose that $0 \leq k < 2^i$ and that $x + 2k$ contains $k$ to base $2$. Consider first the case where $0 \leq k < 2^{i-1}$. Here since $x + 2k \equiv 2k \mod{2^i}$, we have that $2k$ contains $k$ to base $2$, which can only happen if $k = 0$.

Consider then $2^{i-1} \leq k < 2^i$ and write $k = 2^{i-1} + k'$, where $0 \leq k' < 2^{i-1}$. Then $x + 2k \equiv 2k' \mod{2^i}$, so $2k'$ contains $k = 2^{i-1} + k'$ to base $2$. But then $2k'$ must also contain $k'$ to base $2$, so $k' = 0$ and $k = 2^{i-1}$. In this case $x+2k \equiv 2^i + 2^i \equiv 0 \mod{2^{i+1}}$, so $x + 2k$ does not contain $k$ to base $2$, contradiction. 
\end{proof}

\begin{lemma}\label{arithmeticlemma}
Let $x \geq 2^{i+1}$, where $i \geq 0$. Suppose that $x \equiv 2^i + t \mod{2^{i+1}}$, where $0 \leq t < 2^i$. If $0 \leq 2j \leq t$ and $x-2j$ contains $2^i - j$ to base $2$, then $j = 0$.
\end{lemma}

\begin{proof}
We prove the claim by induction on $i$. If $i = 0$ or $i = 1$, then the claim is immediate since $0 \leq 2j \leq t < 2$. Suppose then that $i > 1$. Assume that $0 < 2j \leq t$ and that $x-2j$ contains $2^i - j$ to base $2$. Now $2j \leq t < 2^i$, so $0 < j < 2^{i-1}$. Therefore $2^{i-1}$ must occur in the binary expansion of $2^i - j = 2^{i-1} + (2^{i-1} - j)$, so by our assumption $2^{i-1}$ occurs in the binary expansion of $x-2j$. Note that this also means that $x - 2j$ contains $2^{i-1} - j$ to base $2$.

Now $x-2j \equiv 2^i + (t-2j) \mod{2^{i+1}}$ and $0 \leq t-2j < 2^i$, so it follows that $2^{i-1}$ will occur in the binary expansion of $t-2j$. Write $t = 2^{i-1} + t'$, where $0 \leq t' < 2^{i-1}$. Here $t' \geq 2j$ because $t-2j \geq 2^{i-1}$. Finally, since $x-2j$ contains $2^{i-1} - j$ in base $2$ and $x \equiv 2^{i-1} + t' \mod{2^i}$, we have $j = 0$ by induction.\end{proof}

Now the following corollaries are immediate from Theorem \ref{premetsuprunenko_thm} and lemmas \ref{firstarithmeticlemma} and \ref{arithmeticlemma}.

\begin{seur}\label{trivialsubmoduleC}
Assume that $\Ch \fieldsymbol = 2$, and let $l \geq 2^{i+1}$, where $i \geq 0$. Suppose that $l + 1 \equiv 2^i \mod{2^{i+1}}$. Then $V(\omega_{2^{i+1}}) = L(\omega_{2^{i+1}}) / L(0)$.
\end{seur}

\begin{proof}For $0 \leq j \leq 2^{i+1}$, by Theorem \ref{premetsuprunenko_thm} the irreducible $L(\omega_j)$ is a composition factor of $V(\omega_{2^{i+1}})$ if and only if $j = 2j'$ and $l+1-2j'$ contains $2^i - j'$ to base $2$. By Lemma \ref{firstarithmeticlemma}, this is equivalent to $j' = 0$ or $j' = 2^i$.\end{proof}

\begin{seur}\label{nontrivialcfweylC}
Assume that $\Ch \fieldsymbol = 2$, and let $l \geq 2^{i+1}$, where $i \geq 0$. Suppose that $l + 1 \equiv 2^i + t \mod{2^{i+1}}$, where $0 \leq t < 2^i$. Then any nontrivial composition factor of $V(\omega_{2^{i+1}})$ has the form $L(\omega_{2j})$, where $2^{i+1} \geq 2j \geq t+1$.
\end{seur}

\begin{proof}For $0 \leq j \leq 2^{i+1}$, by Theorem \ref{premetsuprunenko_thm} the irreducible $L(\omega_j)$ is a composition factor of $V(\omega_{2^{i+1}})$ if and only if $j = 2j'$ and $l+1-2j'$ contains $2^i - j'$ to base $2$. If $0 \leq j \leq t$, then by Lemma \ref{arithmeticlemma} we have $j = 0$.\end{proof}

\subsection{Construction of $V(\omega_r)$}\label{weylconstructionC}

We now describe the well known construction of the Weyl modules $V(\omega_r)$ for $G$ using the exterior algebra of the natural module. We will consider our group $G$ as a Chevalley group constructed from a complex simple Lie algebra of type $C_l$. For details of the Chevalley group construction see \cite{SteinbergNotes}.


Let $e_1, \ldots, e_{l}, e_{-l}, \ldots, e_{-1}$ be a basis for a complex vector space $V_\C$, and let $V_\Z$ be the $\Z$-lattice spanned by this basis. We have a non-degenerate alternating form $(-,-)$ on $V_\C$ defined by $(e_i, e_{-i}) = 1 = -(e_{-i}, e_i)$ and $(e_i, e_j) = 0$ for $i \neq -j$. Let $\mathfrak{sp}(V_\C)$ be the Lie algebra formed by the linear endomorphisms $X$ of $V_\C$ satisfying $(Xv, w) + (v, Xw) = 0$ for all $v, w \in V_\C$. Then $\mathfrak{sp}(V_\C)$ is a simple Lie algebra of type $C_l$. Let $\mathfrak{h}$ be the Cartan subalgebra formed by the diagonal matrices in $\mathfrak{sp}(V_\C)$. Then $\mathfrak{h} = \{ \operatorname{diag}(h_1, \ldots, h_l, -h_l, \ldots, -h_1) : h_i \in \C \}$. For $1 \leq i \leq l$, define maps $\varepsilon_i : \mathfrak{h} \rightarrow \C$ by $\varepsilon_i(h) = h_i$ where $h$ is a diagonal matrix with diagonal entries $(h_1, \ldots, h_l, -h_l, \ldots, -h_1)$. Now $\Phi = \{ \pm(\varepsilon_i \pm \varepsilon_j) : 1 \leq i < j \leq l \} \cup \{ \pm 2\varepsilon_i : 1 \leq i \leq l \}$ is the root system for $\mathfrak{sp}(V_\C)$, $\Phi^+ = \{ \varepsilon_i \pm \varepsilon_j : 1 \leq i < j \leq l \} \cup \{ 2 \varepsilon_i : 1 \leq i \leq l \}$ is a system of positive roots, and $\Delta = \{ \varepsilon_i - \varepsilon_{i+1} : 1 \leq i < l \} \cup \{2 \varepsilon_l \}$ is a base for $\Phi$. 

For any $i,j$ let $E_{i,j}$ be the linear endomorphism on $V_\C$ such that $E_{i,j}(e_j) = e_i$ and $E_{i,j}(e_k) = 0$ for $k \neq j$. Then a Chevalley basis for $\mathfrak{sp}(V_\C)$ is given by $X_{\varepsilon_i - \varepsilon_j} = E_{i,j} - E_{-j,-i}$ for all $i \neq j$, by $X_{\pm (\varepsilon_i + \varepsilon_j)} = E_{\pm j, \mp i} + E_{\pm i, \mp j}$ for all $i \neq j$, by $X_{\pm 2 \varepsilon_i} = E_{\pm i, \mp i}$ for all $i$, and by $H_{\varepsilon_i - \varepsilon_{i+1}} = E_{i,i} - E_{-i,-i}$, $H_{2 \varepsilon_l} = E_{l,l} - E_{-l,-l}$.

Let $\mathscr{U}_\Z$ be the Kostant $\Z$-form with respect to this Chevalley basis of $\mathfrak{sp}(V_\C)$. That is, $\mathscr{U}_\Z$ is the subring of the universal enveloping algebra of $\mathfrak{sp}(V_\C)$ generated by $1$ and all $\frac{X_{\alpha}^k}{k!}$ for $\alpha \in \Phi$ and $k \geq 1$. 

Now $V_\Z$ is a $\mathscr{U}_\Z$-invariant lattice in $V_\C$. We define $V = V_\Z \otimes_\Z \fieldsymbol$. Note that $(-,-)$ also defines a non-degenerate alternating form on $V$. Then the simply connected Chevalley group of type $C_l$ induced by $V$ is equal to the group $G = \Sp(V)$ of invertible linear maps preserving $(-,-)$ \cite[pg. 396-397]{Ree}. By abuse of notation we identify the basis $(e_i \otimes 1)$ of $V$ with $(e_i)$. 

Note that for all $1 \leq k \leq 2l$, the Lie algebra $\mathfrak{sp}(V_\C)$ acts naturally on $\wedge^k(V_\C)$ by $$X \cdot (v_1 \wedge \cdots \wedge v_k) = \sum_{i = 1}^k v_1 \wedge \cdots \wedge v_{i-1} \wedge Xv_i \wedge v_{i+1} \wedge \cdots \wedge v_k$$ for all $X \in \mathfrak{sp}(V_\C)$ and $v_i \in V_\C$. With this action, the $\Z$-lattice $\wedge^k(V_\Z)$ is invariant under $\mathscr{U}_\Z$ and this induces an action of $G$ on $\wedge^k(V_\Z) \otimes_\Z \fieldsymbol$. One can show that $g \cdot (v_1 \wedge \cdots \wedge v_k) = gv_1 \wedge \cdots \wedge gv_k$ for all $g \in G$ and $v_i \in V$, so we can and will identify $\wedge^k(V_\Z) \otimes_\Z \fieldsymbol$ and $\wedge^k(V)$ as $G$-modules.

The diagonal matrices in $G$ form a maximal torus $T$. Then a basis of weight vectors of $\wedge^k(V)$ is given by the elements $e_{i_1} \wedge \cdots \wedge e_{i_k}$, where $-l \leq i_1 < \cdots < i_k \leq l$. The basis vector $e_1 \wedge \cdots \wedge e_k$ has weight $\omega_k$.

The form on $V$ induces a form on the exterior power $\wedge^k(V)$ by $$\langle v_1 \wedge \cdots \wedge v_k, w_1 \wedge \cdots \wedge w_k \rangle = \det ((v_i, w_j))_{1 \leq i,j \leq k}$$ for all $v_i, w_j \in V$ \cite[§1, Définition 12, pg. 30]{Bourbaki9}. This form on $\wedge^k(V)$ is invariant under the action of $G$ since $(-,-)$ is. Furthermore, let $e_{i_1} \wedge \cdots \wedge e_{i_k}$ and $e_{j_1} \wedge \cdots \wedge e_{j_k}$ be two basis elements of $\wedge^k(V)$. Then $$\langle e_{i_1} \wedge \cdots \wedge e_{i_k}, e_{j_1} \wedge \cdots \wedge e_{j_k} \rangle = \begin{cases} \pm 1, & \mbox{if } \{i_1, \cdots, i_k\} = \{-j_1, \cdots, -j_k\}. \\ 0, & \mbox{otherwise.} \end{cases} $$ Therefore it follows that the form $\langle -,- \rangle$ on $\wedge^k(V)$ is nondegenerate if $1 \leq k \leq l$. In precisely the same way we find a basis of weight vectors for $\wedge^k(V_\Z)$ and define a form $\langle -,- \rangle_\Z$ on $\wedge^k(V_\Z)$. Note that $\langle -,- \rangle$, $\langle -,- \rangle_\Z$ are alternating if $k$ is odd and symmetric if $k$ is even.

It is well known that there is a unique submodule of $\wedge^k(V)$ isomorphic to the Weyl module $V(\omega_k)$ of $G$, as shown by the following lemma. The following lemma is also a consequence of \cite[4.9]{AndersenJantzen}.

\begin{lemma}\label{weyllemma}
Let $1 \leq k \leq l$, and let $W$ be the $G$-submodule of $\wedge^k(V)$ generated by $e_1 \wedge \cdots \wedge e_k$. Then

\begin{enumerate}[\normalfont (i)]
\item $W$ is equal to the subspace of $\wedge^k(V)$ spanned by all $v_1 \wedge \cdots \wedge v_k$, where $\langle v_1, \cdots, v_k \rangle$ is a $k$-dimensional totally isotropic subspace of $V$. Furthermore, $\dim W = \binom{2l}{k} - \binom{2l}{k-2}$.
\item $W$ is isomorphic to the Weyl module $V(\omega_k)$.
\end{enumerate}
\end{lemma}

\begin{proof}
\begin{enumerate}[(i)]
\item Since $G$ acts transitively on the set of $k$-dimensional totally isotropic subspaces of $V$, it follows that $W$ is spanned by all $v_1 \wedge \cdots \wedge v_k$, where $\langle v_1, \ldots, v_k \rangle$ is a $k$-dimensional totally isotropic subspace of $V$. Then the claim about the dimension of $W$ follows from a result proven for example in \cite[Theorem 1.1]{DeBruyn2}, \cite[Theorem 1.1]{Brouwer} or (in odd characteristic) \cite[pg. 1337]{PremetSuprunenko}.

\item Since $e_1 \wedge \cdots \wedge e_k$ is a maximal vector of weight $\omega_k$ for $G$, the submodule $W$ generated by it is an image of $V(\omega_k)$ \cite[II.2.13]{JantzenBook}. Now $\dim V(\omega_k) = \binom{2l}{k} - \binom{2l}{k-2}$ \cite[Ch. VIII, 13.3, pg. 203]{Bourbaki}, so by (i) $W$ must be isomorphic to $V(\omega_k)$. 
\end{enumerate}\end{proof}

In what follows we will identify $V(\omega_k)$ as the submodule $W$ of $\wedge^k(V)$ given by Lemma \ref{weyllemma}. Set $V(\omega_k)_\Z = \mathscr{U}_\Z (e_1 \wedge \cdots \wedge e_k)$. Note that now we can (and will) identify $V(\omega_k)$ and $V(\omega_k)_\Z \otimes_\Z \fieldsymbol$.

We will denote $y_i = e_i \wedge e_{-i}$ for all $1 \leq i \leq l$. Then if $k = 2s$ is even, a basis for the zero weight space of $\wedge^k(V)$ is given by vectors of the form $y_{i_1} \wedge \cdots \wedge y_{i_s}$, where $1 \leq i_1 < \cdots < i_s \leq l$. There is also a description of a basis for the zero weight space of $V(\omega_k)_\Z$ in \cite[Lemma 10, pg.43]{JantzenThesis}. For our purposes, we will only need a convenient set of generators given by the next lemma.

%

\begin{lemma}[{\cite[pg. 40, Lemma 6]{JantzenThesis}}]\label{weightzerolemma}
Suppose that $k$ is even, say $k = 2s$, where $1 \leq k \leq l$. Then the zero weight space of $V(\omega_k)_\Z$ (thus also of $V(\omega_k)$) is spanned by vectors of the form $$(y_{j_1} - y_{k_1}) \wedge \cdots \wedge (y_{j_s} - y_{k_s}),$$ where $1 \leq k_r < j_r \leq l$ for all $r$ and $j_r, k_r \neq j_{r'}, k_{r'}$ for all $r \neq r'$.
\end{lemma}

\begin{lemma}\label{wedgefixedpoint}
Suppose that $k$ is even, say $k = 2s$, where $1 \leq k \leq l$. Then the vector $$\gamma = \sum_{1 \leq i_1 < \cdots < i_s \leq l} y_{i_1} \wedge \cdots \wedge y_{i_s}$$ is fixed by the action of $G$ on $\wedge^k(V)$. Furthermore, any $G$-fixed point in $\wedge^k(V)$ is a scalar multiple of $\gamma$.
\end{lemma}

\begin{proof}
To see that $\gamma$ is fixed by $G$, see for example \cite[3.4]{DeBruyn} where it is shown that the definition of $\gamma$ does not depend on the symplectic basis chosen.

For the other claim, note first that any $G$-fixed point must have weight zero. Recall that the zero weight space of $\wedge^k(V)$ has basis $$\mathscr{B} = \{y_{i_1} \wedge \cdots \wedge y_{i_s} : 1 \leq i_1 < \cdots < i_s \leq l \}.$$ Now the group $\Sigma_l$ of permutations of $\{1,2, \ldots, l\}$ acts on $V$ by $\sigma \cdot e_{\pm i} = e_{\pm \sigma(i)}$ for all $\sigma \in \Sigma_l$. Clearly this action preserves the form $(-,-)$ on $V$, so this gives an embedding $\Sigma_l < G$. Note also that $\Sigma_l$ acts transitively on $\mathscr{B}$. Thus any $\Sigma_l$-fixed point in the linear span of $\mathscr{B}$ must be a scalar multiple of $\sum_{b \in \mathscr{B}} b = \gamma$.
\end{proof}

With these preliminary steps done, we now move on to proving Proposition \ref{fundamentalquadC}. For the rest of this section, we will make the following assumption. 

\begin{center}\emph{Assume that $\Ch \fieldsymbol = 2$.}\end{center}

Let $1 \leq r \leq l$. By Proposition \ref{naturalquad}, we know that $L(\omega_1)$ is not orthogonal. Suppose then that $r \geq 2$ and that $L(\omega_r)$ is not orthogonal. By Proposition \ref{char2quad} (iii) we have $\operatorname{H}^1(G, L(\omega_r)) \neq 0$, so by Corollary \ref{extcorollary} we have $r = 2^{i+1}$ for some $i \geq 0$ and $l+1 \equiv 2^i + t \mod{2^{i+1}}$ for some $0 \leq t < 2^i$. What remains is to determine when $L(\omega_r)$ is orthogonal for such $r$. With the lemma below, we reduce this to the evaluation of $Q(v)$ for a single vector $v \in V(\omega_r)$, where $Q$ is a non-zero $G$-invariant quadratic form on $V(\omega_r)$.

\begin{lemma}\label{mainlemma}
Let $l \geq 2^{i+1}$, where $i \geq 0$. Suppose that $l + 1 \equiv 2^i + t \mod{2^{i+1}}$, where $0 \leq t < 2^i$. Define the vector $\gamma \in \wedge^{2^{i+1}}(V)$ to be equal to $$\sum_{1 \leq i_1 < \cdots < i_{2^i} \leq l-t} y_{i_1} \wedge \cdots \wedge y_{i_{2^i}}.$$ Then 

\begin{enumerate}[\normalfont (i)]
\item $\gamma$ is in $V(\omega_{2^{i+1}})$ and is a fixed point for the subgroup $C_{l-t} < G$,
\item $\gamma$ is in $\rad V(\omega_{2^{i+1}})$,
\item $L(\omega_{2^{i+1}})$ is orthogonal if and only if $Q(\gamma) = 0$, where $Q$ is a nonzero $G$-invariant quadratic form on $V(\omega_{2^{i+1}})$.
\end{enumerate}

\end{lemma}

\begin{proof}

\begin{enumerate}[(i)]

\item By Lemma \ref{wedgefixedpoint}, $\gamma$ is a fixed by the action of $C_{l-t}$. It follows from Lemma \ref{weyllemma} that the $C_{l-t}$-submodule $W$ generated by $e_1 \wedge \cdots \wedge e_{2^{i+1}}$ is isomorphic to the Weyl module $V_{C_{l-t}}(\omega_{2^{i+1}})$. By Corollary \ref{trivialsubmoduleC}, the module $W$ has a trivial $C_{l-t}$-submodule and by Lemma \ref{wedgefixedpoint} it is generated by $\gamma$. Since $W$ is contained in $V(\omega_{2^{i+1}})$, the $C_{l}$-submodule generated by $e_1 \wedge \cdots \wedge e_{2^{i+1}}$, the claim follows. 

For a different proof, one can also prove $\gamma \in V(\omega_{2^{i+1}})$ by showing that $\gamma$ is in the kernel of certain linear maps as defined in \cite[Theorem 3.5]{DeBruyn} or \cite[Theorem 3.1, Proposition 3.3]{Brouwer}.

\item By Proposition \ref{char2quad} (ii), it will be enough to show that $\gamma$ is orthogonal to $V(\omega_{2^{i+1}})$ with respect to the form $\langle -,- \rangle$ on $\wedge^{2^{i+1}}(V)$. Since $\gamma$ has weight $0$, it is orthogonal to any vector of weight $\neq 0$. Therefore it suffices to show that $\gamma$ is orthogonal to any vector of weight $0$ in $V(\omega_{2^{i+1}})$. By Lemma \ref{weightzerolemma}, this follows once we show that $\gamma$ is orthogonal to any vector $$\delta = (y_{j_1} + y_{k_1}) \wedge \cdots \wedge (y_{j_{2^i}} + y_{k_{2^i}}),$$ where $1 \leq k_r < j_r \leq l$ for all $r$ and $k_r, j_r \neq j_{r'}, k_{r'}$ for all $r \neq r'$.

Because the set $\{j_1, \ldots, j_{2^i}\}$ contains $2^i$ distinct integers, we cannot have $j_r \geq l-t+1$ for all $r$. Indeed, otherwise $j_r \geq l-t+2^i \geq l+1$ for some $r$, contradicting the fact that $j_r \leq l$. Let $q > 0$ be the number of $j_r$ such that $j_r \leq l-t$.

The vector $\delta$ can be written as $$\sum_{f_s \in \{j_s, k_s\}} y_{f_1} \wedge \cdots \wedge y_{f_{2^i}}.$$ Now $$\langle y_{f_1} \wedge \cdots \wedge y_{f_{2^i}}, y_{g_1} \wedge \cdots \wedge y_{g_{2^i}} \rangle = \begin{cases} 1, & \mbox{if } \{f_1, \ldots, f_{2^i}\} = \{g_1, \ldots, g_{2^i}\}. \\ 0, & \mbox{otherwise.} \end{cases} $$ and so $\langle \delta, \gamma \rangle$ is an integer, equal to the number of $y_{f_1} \wedge \cdots \wedge y_{f_{2^i}}$ in the sum such that $f_{s} \leq l-t$ for all $1 \leq s \leq 2^i$. Thus if $j_r \geq k_r \geq l-t+1$ for some $r$, then $\langle \gamma, \delta \rangle = 0$. If $k_r \leq l-t$ for all $r$, then it follows that $\langle \delta, \gamma \rangle = 2^{q} = 0$ since $q > 0$.

\item By Lemma \ref{extcorollary} we have $\operatorname{H}^1(G, L(\omega_{2^{i+1}})) \neq 0$ and so there exists a nonsplit extension of $L(\omega_{2^{i+1}})$ by the trivial module $\fieldsymbol$. We can find this extension as an image of the Weyl module $V(\omega_{2^{i+1}})$ \cite[II.2.13, II.2.14]{JantzenBook}, so $\rad V(\omega_{2^{i+1}}) / M \cong \fieldsymbol$ for some submodule $M$ of $\rad V(\omega_{2^{i+1}})$. Since each composition factor of $V(\omega_{2^{i+1}})$ occurs with multiplicity one (Theorem \ref{premetsuprunenko_thm}), each composition factor of $M$ is nontrivial. Then by Corollary \ref{nontrivialcfweylC} and Lemma \ref{restrictionlemmaC} (ii), the restriction $M \downarrow C_{l-t}$ has no trivial composition factors. But by (i) $\gamma$ is a fixed point for $C_{l-t}$, so it follows that $\gamma \not\in M$ and then $\rad V(\omega_{2^{i+1}}) = \langle \gamma \rangle \oplus M$ as $C_{l-t}$-modules. 

Now let $Q$ be a nonzero $G$-invariant quadratic form on $V(\omega_{2^{i+1}})$. Since for the polarization $b_Q$ of $Q$ we have $\rad b_Q = \rad V(\omega_{2^{i+1}})$ (Proposition \ref{char2quad} (ii)), composing $Q$ with the square root map $\fieldsymbol \rightarrow \fieldsymbol$ defines a morphism $\rad V(\omega_{2^{i+1}}) \rightarrow \fieldsymbol$ of $G$-modules. Therefore $Q$ must vanish on $M$, since $M$ has no trivial composition factors. Thus for all $m \in M$ and scalars $c$ we have $Q(c \gamma + m) = c^2 Q(\gamma)$, so $Q$ vanishes on $\rad V(\omega_{2^{i+1}})$ if and only if $Q(\gamma) = 0$. Hence by Proposition \ref{char2quad} (iii) $L(\omega_{2^{i+1}})$ is orthogonal if and only if $Q(\gamma) = 0$. 


\end{enumerate}

\end{proof}


\subsection{Computation of a quadratic form $Q$ on $V(\omega_r)$}\label{typeCquadcomp}

To finish the proof of Proposition \ref{fundamentalquadC} we still have to compute $Q(\gamma)$ for the vector $\gamma$ from Lemma \ref{mainlemma}.

We retain the notation from the previous subsection and keep the assumption that $\Ch \fieldsymbol = 2$. Let $r$ be even, say $r = 2s$, where $1 \leq r \leq l$. Now the form $\langle -,- \rangle_\Z$ on $\wedge^r(V_\Z)$ induces a quadratic form $q_\Z$ on $\wedge^r(V_\Z)$ by $q_\Z(x) = \langle x, x \rangle$. We will use this form to find a nonzero $G$-invariant quadratic form on $V(\omega_r) = V(\omega_r)_\Z \otimes_\Z \fieldsymbol$.


\begin{lemma}\label{quadevenlemma}
We have $q_\Z(V(\omega_r)_\Z) \subseteq 2 \Z$ and $q_\Z(V(\omega_r)_\Z) \not\subseteq 4 \Z$.
\end{lemma}

\begin{proof}
Let $\alpha = e_1 \wedge \cdots \wedge e_r$ and $\beta = e_{-1} \wedge \cdots \wedge e_{-r}$. Now $\alpha, \beta \in V(\omega_r)_\Z$ and $\langle \alpha, \alpha \rangle = \langle \beta, \beta \rangle = 0$ and $\langle \alpha, \beta \rangle = 1$, giving $q_\Z(\alpha + \beta) = 2$ and so $q_\Z(V(\omega_r)_\Z) \not\subseteq 4 \Z$. If we had $q_\Z(V(\omega_r)_\Z) \not\subseteq 2 \Z$, then $Q = q_\Z \otimes_\Z \fieldsymbol$ defines a nonzero $G$-invariant quadratic form on $V(\omega_r)$. But then the polarization of $Q$ is equal to $2 \langle -,- \rangle = 0$, which by Proposition \ref{char2quad} (i) and (ii) is not possible.\end{proof}

Therefore $Q = \frac{1}{2} q_\Z \otimes_\Z \fieldsymbol$ defines a nonzero $G$-invariant quadratic form on $V(\omega_r)$ with polarization $\langle -,- \rangle$. A similar construction when $r$ is odd is discussed in \cite[Proposition 8.1]{GaribaldiNakano}. 

Now if we consider a zero weight vector of the form $$\sum_{\{i_1, \cdots, i_s\} \in \mathscr{I}} y_{i_1} \wedge \cdots \wedge y_{i_s}$$ in $V(\omega_{r})$, the value of $Q$ for this vector is equal to $\frac{|\mathscr{I}|}{2}$ since $$\langle y_{f_1} \wedge \cdots \wedge y_{f_k}, y_{g_1} \wedge \cdots \wedge y_{g_k} \rangle = \begin{cases} 1, & \mbox{if } \{f_1, \ldots, f_k\} = \{g_1, \ldots, g_k\}. \\ 0, & \mbox{otherwise.} \end{cases}$$ Now we can compute the value of $Q(\gamma)$ from Lemma \ref{mainlemma}. Since there are $\binom{l-t}{2^i}$ terms occurring in the sum that defines $\gamma$, we have $Q(\gamma) = \frac{1}{2} \binom{l-t}{2^i}$. Thus $Q(\gamma) = 0$ if and only if $\binom{l-t}{2^i}$ is divisible by $4$. Now the proof of Proposition \ref{fundamentalquadC} is finished with the following lemma.

\begin{lemma}\label{binomiallemma}
Let $l + 1 \equiv 2^i + t \mod{2^{i+1}}$, where $0 \leq t < 2^i$. The integer $\binom{l-t}{2^i}$ is divisible by $4$ if and only if $l + 1 \equiv 2^i + t \mod{2^{i+2}}$.
\end{lemma}

\begin{proof}
According to Kummer's theorem, if $p$ is prime and $d \geq 0$ is maximal such that $p^d$ divides $\binom{x}{y}$ ($x \geq y \geq 0$), then $d$ is the number of carries that occur when adding $y$ to $x-y$ in base $p$. Now $(l-t)-2^i \equiv -1 \equiv 2^i + \cdots + 2 + 1 \mod{2^{i+1}}$. If $$l-t-2^i \equiv 2^i + \cdots + 2 + 1 \mod{2^{i+2}},$$ then adding $2^i$ to $l-t-2^i$ in binary results in just one carry. The other possibility is that $$l-t-2^i \equiv 2^{i+1} + 2^i + \cdots + 2 + 1 \mod{2^{i+2}},$$ and in this case there are $\geq 2$ carries. Therefore $\binom{l-t}{2^i}$ is divisible by $4$ if and only if $$l-t-2^i \equiv 2^{i+1} + 2^i + \dots + 2 + 1 \equiv -1 \mod{2^{i+2}},$$ which is equivalent to $l+1 \equiv 2^i + t \mod{2^{i+2}}$.\end{proof}

\section{Fundamental irreducibles for classical types}\label{section:classicalfundamental}

\noindent With a bit more work, we can use Proposition \ref{fundamentalquadC} to determine for all classical types the fundamental irreducible representations that are orthogonal. In this section assume that $\Ch \fieldsymbol = 2$.

For a groups of type $A_l$ ($l \geq 1$), the only self-dual fundamental irreducible representations are those of form $L(\omega_{\frac{l+1}{2}})$, where $l$ is odd. Furthermore, all fundamental representations are minuscule, so $V(\omega_{\frac{l+1}{2}}) = L(\omega_{\frac{l+1}{2}})$ and thus by Proposition \ref{char2quad} (iv) the representation $L(\omega_{\frac{l+1}{2}})$ is orthogonal.

Now for type $B_l$, there exists an exceptional isogeny $\varphi: B_l \rightarrow C_l$ between simply connected groups of type $B_l$ and $C_l$ \cite[Theorem 28]{SteinbergNotes}. Then irreducible representations of $C_l$ induce irreducible representations $B_l$ by twisting with the isogeny $\varphi$. For fundamental irreducible representations, we have $L_{C_l}(\omega_r)^\varphi \cong L_{B_l}(\omega_r)$ if $1 \leq r \leq l-1$, and $L_{C_l}(\omega_l)^\varphi$ is a Frobenius twist of $L_{B_l}(\omega_l)$. Therefore for all $1 \leq r \leq l$, the representation $L_{C_l}(\omega_r)$ is orthogonal if and only if $L_{B_l}(\omega_r)$ is orthogonal.

Consider then type $D_l$ ($l \geq 4$). First note that the natural representation $L_{D_l}(\omega_1)$ of $D_l$ is orthogonal. Now since we are working in characteristic two, there is an embedding $D_l < C_l$ as a subsystem subgroup generated by the short root subgroups. Then if $1 \leq r \leq l-2$, we have $L_{C_l}(\omega_r) \downarrow D_l \cong L_{D_l}(\omega_r)$ for $1 \leq r \leq {l-2}$ by \cite[Theorem 4.1]{SeitzClassical}. By combining this fact with the lemma below, we see for $2 \leq r \leq l-2$ that $L_{C_l}(\omega_r)$ is orthogonal if $L_{D_l}(\omega_r)$ is orthogonal. 


\begin{lemma}\label{lemma:typeDrestriction}
Let $G$ be simple of type $C_l$ and consider $H < G$ of type $D_l$ as the subsystem subgroup generated by short root subgroups. Suppose that $V$ is a nontrivial irreducible $2$-restricted representation of $G$ and $V \neq L_G(\omega_1)$. Then if $V \downarrow H$ is $2$-restricted irreducible, the representation $V$ is orthogonal for $G$ if and only if $V$ is orthogonal for $H$.
\end{lemma}

\begin{proof}[Proof (G. Seitz).]

 If $V$ is an orthogonal $G$-module, it is clear that it is an orthogonal $H$-module as well. Suppose then that $V$ is not orthogonal for $G$. Since $V$ is not the natural module for $G$, by Proposition \ref{char2quad} (iii) there exists a nonsplit extension $$0 \rightarrow \langle w \rangle \rightarrow M \rightarrow V \rightarrow 0$$ of $G$-modules, where $w \in M$. Furthermore, there exists a nonzero $G$-invariant quadratic form $Q$ on $M$ such that $Q(w) \neq 0$.

We claim that $M \downarrow H$ is also a nonsplit extension. If this is not the case, then $M \downarrow H = W \oplus \langle w \rangle$ for some $H$-submodule $W$ of $M$. We will show that $W$ is invariant under $G$, which is a contradiction since $M$ is nonsplit for $G$. Now $W$ is $2$-restricted irreducible for $H$, so by a theorem of Curtis \cite[Theorem 6.4]{BorelLectures} the module $W$ is also an irreducible representation of $\Lie(H)$. Since $\Lie(H)$ is an ideal of $\Lie(G)$ that is invariant under the adjoint action of $G$, it follows that $gW$ is $\Lie(H)$-invariant for all $g \in G$. But as a $\Lie(H)$-module $M$ is the sum of a trivial module and $W$, so we must have $gW = W$ for all $g \in G$.

Thus if $V \downarrow H = L_H(\lambda)$, then there exists a surjection $\pi: V_H(\lambda) \rightarrow M$ of $H$-modules \cite[II.2.13]{JantzenBook}. Now the quadratic form $Q$ induces via $\pi$ a nonzero, $H$-invariant quadratic form on $V_H(\lambda)$ which does not vanish on the radical of $V_H(\lambda)$. By Proposition \ref{char2quad} (iii) the representation $V$ is not orthogonal for $H$.
\end{proof}

Finally, the half-spin representations of $D_l$ are minuscule representations, so $L_{D_l}(\omega_l) = V_{D_l}(\omega_l)$ and $L_{D_l}(\omega_{l-1}) = V_{D_l}(\omega_{l-1})$. As before, by Proposition \ref{char2quad} (iv) it follows that $L_{D_l}(\omega_l)$ and $L_{D_l}(\omega_{l-1})$ are orthogonal if they are self-dual. Therefore we can conclude that for $l = 4$ and $l = 5$ all self-dual $L_{D_l}(\omega_i)$ are orthogonal. 

Note that if $l \geq 6$, then $L_{C_l}(\omega_l)$ and $L_{C_l}(\omega_{l-1})$ are also orthogonal (Example \ref{spinexample}). Thus for $l \geq 6$ we have for all $2 \leq r \leq l$ that $L_{C_l}(\omega_r)$ is orthogonal if $L_{D_l}(\omega_r)$ is orthogonal. 

Taking all of this together, Proposition \ref{fundamentalquadC} is improved to the following.

\begin{lause}\label{mainprop}
Assume that $\Ch \fieldsymbol = 2$. Let $G$ be simple of type $A_l$ ($l \geq 1$), $B_l$ ($l \geq 2$), $C_l$ ($l \geq 2$) or $D_l$ ($l \geq 4$). Suppose $1 \leq r \leq l$ and $\omega_r = -w_0(\omega_r)$. Then $L(\omega_r)$ is not orthogonal if and only if one of the following holds:

\begin{itemize}
\item $G$ is of type $B_l$ ($l \geq 2$) or $C_l$ ($l \geq 2$) and $r = 1$.
\item $G$ is of type $B_l$ ($l \geq 2$), $C_l$ ($l \geq 2)$ or $D_l$ ($l \geq 6$) and $r = 2^{i+1}$ for some $i \geq 0$ such that $l+1 \equiv 2^{i+1} + 2^i + t \mod{2^{i+2}}$, where $0 \leq t < 2^i$.
\end{itemize}

\end{lause}

\section{Representations $L(\omega_r + \omega_s)$ for type $A_l$}

\noindent Assume that $G$ is simply connected of type $A_l$, $l \geq 2$. Set $n = l+1$. 

In this section, we determine when in characteristic $2$ the irreducible representation $L(\omega_r + \omega_s)$, $1 \leq r < s \leq l$, of $G$ has a nonzero $G$-invariant quadratic form. Now $L(\omega_r + \omega_s)$ is not orthogonal if it is not self-dual, so it will be enough to consider $L(\omega_r + \omega_{n-r})$, where $1 \leq r < n-r \leq l$ (see Table \ref{dualitytable}). In this case, the answer and the methods to prove it are very similar to those found in Section \ref{typeCsection}. The result is the following theorem, which we will prove in what follows.

\begin{lause}\label{quadpropA}
Assume $\Ch \fieldsymbol = 2$. Let $1 \leq r < n-r \leq l$. Then $L(\omega_r + \omega_{n-r})$ is not orthogonal if and only if $r = 2^i$ for some $i \geq 0$ and $n+1 \equiv 2^{i+1} + 2^i + t \mod{2^{i+2}}$, where $0 \leq t < 2^i$.
\end{lause}

The following examples follow easily from Theorem \ref{quadpropA} (cf. examples \ref{exampleC2} and \ref{exampleC4}).

\begin{esim}
If $\Ch \fieldsymbol = 2$, then $L(\omega_1 + \omega_l)$ is orthogonal if and only if $n \not\equiv 2 \mod{4}$. This result was also proven in \cite[Theorem 3.4 (b)]{GowWillems1}.
\end{esim}

\begin{esim}
If $\Ch \fieldsymbol = 2$, then $L(\omega_2 + \omega_{l-1})$ is orthogonal if and only if $n \not\equiv 5,6 \mod{8}$.
\end{esim}


\subsection{Representation theory}

\noindent The composition factors and the submodule structure of the Weyl modules $V(\omega_r + \omega_s)$, $1 \leq r < s \leq l$, were determined by Adamovich \cite{AdamovichThesis}. Using her result, Baranov and Suprunenko have given in \cite[Theorem 2.3]{BaranovSuprunenkoA} a description of the set of composition factors, similarly to Theorem \ref{premetsuprunenko_thm}. For $1 \leq r < s \leq l$, define $J_p(r,s)$ be the set of pairs $(r-k,s+k)$, where $n-s,r \geq k \geq 0$ and $s-r+1+2k$ contains $k$ to base $p$. Here we will define $\omega_0 = 0$ and $\omega_n = 0$, so then $L(\omega_0) = L(\omega_n) = L(\omega_0 + \omega_n)$ is the trivial irreducible module and $L(\omega_0 + \omega_r) = L(\omega_r + \omega_n) = L(\omega_r)$. Now \cite[Theorem 2.3]{BaranovSuprunenkoA} gives the following\footnote{Baranov and Suprunenko give the result in terms of $\pi_{i,j} = L(\omega_{(j-i+1)/2} + \omega_{(i+j-1)/2})$, but from $\pi_{y-x+1, x+y} = L(\omega_x + \omega_y)$ we get the formulation in Theorem \ref{typeAweylcf}.}.

\begin{lause}\label{typeAweylcf}
Let $1 \leq r < s \leq l$. Then in the Weyl module $V(\omega_r + \omega_s)$, each composition factor has multiplicity $1$, and the set of composition factors is $\{L(\omega_j + \omega_{j'}) : (j,j') \in J_p(r,s) \}$.
\end{lause}

The result of Kleshchev and Sheth in \cite{KleshchevSheth} \cite[Corollary 3.6]{KleshchevSheth2} about the first cohomology groups for groups of type $A_l$ gives the following (cf. Corollary \ref{extcorollary}).

\begin{lause}\label{extcorollaryA}
Assume that $\Ch \fieldsymbol = 2$. Let $1 \leq r < s \leq l$. Then $\operatorname{H}^1(G, L(\omega_r + \omega_s)) \neq 0$ if and only if $r = 2^i$, $s = n-2^i$ for some $i \geq 0$, and $n+1 \equiv 2^i + t \mod{2^{i+1}}$ for some $0 \leq t < 2^i$.
\end{lause}

Throughout this section we will consider subgroups $A_{l'} < A_l = G$, which are embedded into $G$ as follows. We consider $G = \SL(V)$, where $V$ has basis $e_1, e_2, \ldots, e_{l+1}$. Then for $1 \leq l' < l$, the embedding $A_{l'} < A_l$ is $\SL(V') < \SL(V)$, where $V' \subseteq V$ has basis $e_1, \ldots, e_{l'+1}$ and $\SL(V')$ fixes the basis vectors $e_{l'+ 2}, \ldots, e_{l+1}$. 

Baranov and Suprunenko have determined the submodule structure of the restrictions $L(\omega_r + \omega_s) \downarrow A_{l-1}$ for all $0 \leq r \leq s \leq n$ in their article \cite[Theorem 1.1]{BaranovSuprunenkoA}. As in Section \ref{typeCsection}, for our purposes it will be enough to know which composition factors occur in the restriction. To state the result of Baranov and Suprunenko, we will denote $\pi_{r,s}^l = L_{A_l}(\omega_r + \omega_{l+1-s})$ for all $0 \leq r \leq l+1-s \leq l+1$. We will define $\pi_{r,s}^l = 0$ if $r < 0$, $s < 0$ or $r + s > l+1$. Now the main result of \cite{BaranovSuprunenkoA} gives the following\footnote{Baranov and Supruneko give their result in terms of $L_{i,j}^l = L_{A_l}(\omega_i + \omega_j)$ for $0 \leq i \leq j \leq n$, but replacing $j$ by $n-j$ gives the formulation in Theorem \ref{typeArestriction}.} (cf. Theorem \ref{typeCrestriction}).

\begin{lause}\label{typeArestriction}
Let $0 \leq r \leq n-s \leq n$ and assume that $n \geq 3$. Set $d = \nu_p(n+1-(r+s))$, and $\varepsilon = 0$ if $n + 1 - (r+s) \equiv -p^{d} \mod{p^{d+1}}$ and $\varepsilon = 1$ otherwise. Then the character of $\pi_{r,s}^l \downarrow A_{l-1}$ is given by $$\ch \pi_{r,s-1}^{l-1} + \ch \pi_{r-1,s}^{l-1} + \ch \pi_{r,s}^{l-1} + \left( \sum_{k = 0}^{d-1} 2 \ch \pi_{r-p^k, s-p^k}^{l-1} \right) + \varepsilon \ch \pi_{r-p^d, s-p^d}^{l-1}$$ where the sum in the brackets is zero if $d = 0$. 
\end{lause}

As with Theorem \ref{typeCrestriction}, note that when $\Ch \fieldsymbol = 2$, we always have $\varepsilon = 0$ in Theorem \ref{typeArestriction}. The following applications of theorems \ref{typeAweylcf} and \ref{extcorollaryA} in characteristic two will be needed later.

\begin{lemma}\label{restrictionlemmaA}
Assume that $\Ch \fieldsymbol = 2$, and let $n \geq 2^{i+1}$, where $i \geq 0$. Suppose that $n + 1 \equiv 2^i + t \mod{2^{i+1}}$, where $0 \leq t < 2^i$. Let $0 \leq x \leq 2^i$ and $0 \leq y \leq 2^i$ be such that $2^{i+1} \geq x+y \geq t+1$. Then

\begin{enumerate}[\normalfont (i)]
\item All composition factors of the restriction $\pi_{x,y}^l \downarrow A_{l-1}$ have the form $\pi_{x',y'}^{l-1}$ for some $0 \leq x', y' \leq 2^i$ such that $x' + y' \geq t$.
\item $\pi_{x,y}^l \downarrow A_{l-t}$ has no trivial composition factors.
\end{enumerate}
\end{lemma}

\begin{proof} 
(cf. Lemma \ref{restrictionlemmaC}) If $t = 0$ there is nothing to prove, so suppose that $t \geq 1$. It will be enough to prove (i) as then (ii) will follow by induction on $t$. Let $d = \nu_2(n+1-(x+y))$. Suppose first that $0 \leq d < i + 1$. Then $n+1-(x+y) \equiv t - (x+y) \mod{2^i}$, so $\nu_2(n+1-(x+y)) = \nu_2((x+y)-t)$. By Theorem \ref{typeArestriction}, the composition factors occurring in $\pi_{x,y}^l \downarrow A_{l-1}$ are $\pi_{x,y-1}^{l-1}$, $\pi_{x-1,y}^{l-1}$, $\pi_{x,y}^{l-1}$, and $\pi_{x-2^k, y-2^k}^{l-1}$ for $0 \leq k \leq d-1$. Therefore the claim follows since $\nu_2((x+y)-t) = d$ and thus $x+y - 2^d \geq t$. 

Consider then the case where $d \geq i+1$. Then $n+1-(x+y) \equiv 2^i + t - (x+y) \equiv 0 \mod{2^{i+1}}$, so $(x+y)-t \equiv 2^i \mod{2^{i+1}}$. On the other hand $0 \leq (x+y)-t < 2^{i+1}$, so $(x+y)-t = 2^i$. By Theorem \ref{typeArestriction} the composition factors occurring in $\pi_{x,y}^l \downarrow A_{l-1}$ are $\pi_{x,y-1}^{l-1}$, $\pi_{x-1,y}^{l-1}$, $\pi_{x,y}^{l-1}$, and $\pi_{x-2^k, y-2^k}^{l-1}$ for $0 \leq k \leq i-1$ (since $x-2^k < 0$ or $y-2^k < 0$ for $i \leq k \leq d-1$), so again the claim follows.\end{proof}

As a consequence of Theorem \ref{typeAweylcf} and lemmas \ref{firstarithmeticlemma} and \ref{arithmeticlemma}, we get the following (cf. corollaries \ref{trivialsubmoduleC} and \ref{nontrivialcfweylC}).

\begin{seur}\label{trivialsubmoduleA}
Assume that $\Ch \fieldsymbol = 2$ and let $n > 2^{i+1}$, where $i \geq 0$. Suppose that $n + 1 \equiv 2^i \mod{2^{i+1}}$. Then $V(\omega_{2^i} + \omega_{n-2^i}) = L(\omega_{2^i} + \omega_{n-2^i}) / L(0)$.
\end{seur}

\begin{proof}
According to Theorem \ref{typeAweylcf}, the composition factors of $V(\omega_{2^i} + \omega_{n-2^i})$ are $L(\omega_{2^i - k} + \omega_{n-2^i + k})$, where $0 \leq k \leq 2^i$ and $n + 1 - 2^{i+1} + 2k$ contains $k$ to base $2$. We can replace $k$ by $2^i - k$, and then the condition is equivalent to $n+1 - 2k$ containing $2^i - k$ to base $2$, which implies $k = 0$ or $k = 2^i$ by Lemma \ref{firstarithmeticlemma}.
\end{proof}

\begin{seur}\label{nontrivialcfweylA}
Assume that $\Ch \fieldsymbol = 2$ and let $n > 2^{i+1}$, where $i \geq 0$. Suppose that $n + 1 \equiv 2^i + t \mod{2^{i+1}}$, where $i \geq 0$ and $0 \leq t < 2^i$. Then any nontrivial composition factor of $V(\omega_{2^i} + \omega_{n - 2^i})$ has the form $L(\omega_x + \omega_{n-y})$ for some $0 \leq x \leq n-y \leq n$ and $x + y \geq t+1$.
\end{seur}

\begin{proof}
According to Theorem \ref{typeAweylcf}, the composition factors of $V(\omega_{2^i} + \omega_{n-2^i})$ are $L(\omega_{2^i - k} + \omega_{n-2^i + k})$, where $0 \leq k \leq 2^i$ and $n + 1 - 2^{i+1} + 2k$ contains $k$ to base $2$. Setting $k' = 2^i - k$, the composition factors are $L(\omega_{k'} + \omega_{n-k'})$, where $n + 1 - 2k'$ contains $2^i - k'$ to base $2$. By Lemma \ref{arithmeticlemma} we have $k' = 0$ or $2k' \geq t+1$, which proves the claim.
\end{proof}

\subsection{Construction of $V(\omega_r + \omega_{n-r})$}

\noindent We now describe a construction of $V(\omega_r + \omega_{n-r})$, in many ways similar to that of $V_{C_l}(\omega_r)$ described in Section \ref{weylconstructionC}. We will consider our group $G$ as a Chevalley group constructed from a complex simple Lie algebra of type $A_l$. 

Let $e_1, e_2, \ldots, e_{n}$ be a basis for a complex vector space $V_\C$, and let $V_\Z$ be the $\Z$-lattice spanned by this basis. Let $\mathfrak{sl}(V_\C)$ be the Lie algebra formed by the linear endomorphisms of $V_\C$ with trace zero. Then $\mathfrak{sl}(V_\C)$ is a simple Lie algebra of type $A_l$. Let $\mathfrak{h}$ be the Cartan subalgebra formed by the diagonal matrices in $\mathfrak{sl}(V_\C)$ (with respect to the basis $(e_i)$). For $1 \leq i \leq n$, define maps $\varepsilon_i : \mathfrak{h} \rightarrow \C$ by $\varepsilon_i(h) = h_i$ where $h$ is a diagonal matrix with diagonal entries $(h_1, h_2, \ldots, h_{n})$. Now $\Phi = \{ \varepsilon_i - \varepsilon_j : i \neq j \}$ is the root system for $\mathfrak{sl}(V_\C)$ and $\Phi^+ = \{ \varepsilon_i - \varepsilon_j : i < j \}$ is a system of positive roots, and $\Delta = \{ \varepsilon_i - \varepsilon_{i+1} : 1 \leq i \leq l \}$ is a base for $\Phi$. 

For any $i,j$ let $E_{i,j}$ be the linear endomorphism on $V_\C$ such that $E_{i,j}(e_j) = e_i$ and $E_{i,j}(e_k) = 0$ for $k \neq j$. Now a Chevalley basis for $\mathfrak{sl}(V_\C)$ is given by $X_{\varepsilon_i - \varepsilon_j} = E_{i,j}$ for $i \neq j$ and $H_{\varepsilon_i - \varepsilon_{i+1}} = E_{i,i} - E_{i+1,i+1}$ for $1 \leq i \leq l$. Let $\mathscr{U}_\Z$ be the Kostant $\Z$-form with respect to this Chevalley basis of $\mathfrak{sl}(V_\C)$. That is, $\mathscr{U}_\Z$ is the subring of the universal enveloping algebra of $\mathfrak{sl}(V_\C)$ generated by $1$ and all $\frac{X_{\alpha}^k}{k!}$ for $\alpha \in \Phi$ and $k \geq 1$.

Now $V_\Z$ is a $\mathscr{U}_\Z$-invariant lattice in $V_\C$. We define $V = V_\Z \otimes_\Z \fieldsymbol$. Then the simply connected Chevalley group of type $A_l$ induced by $V$ is equal to $G = \SL(V)$. 

Let $e_1^*, e_2^*, \ldots, e_{n}^*$ be a basis for $V_\C^*$, dual to the basis $(e_1, e_2, \ldots, e_{n})$ of $V_\C$ (so here $e_i^*(e_j) = \delta_{ij}$). Denote the $\Z$-lattice spanned by $e_1^*, e_2^*, \ldots, e_{n}^*$ by $V_\Z^*$. Then $V_\Z^*$ is $\mathscr{U}_\Z$-invariant and we can identify $V_\Z^* \otimes_\Z \fieldsymbol$ and $V^*$ as $G$-modules. Here the action of $G$ on $V^*$ is given by $(g \cdot f)(v) = f(g^{-1} v)$ for all $g \in G$, $f \in V^*$ and $v \in V$. 

By abuse of notation we identify the basis $(e_i \otimes 1)$ of $V$ with $(e_i)$, and the basis $(e_i^* \otimes 1)$ of $V^*$ with $(e_i^*)$.

Let $1 \leq k < n-k \leq l$. Now the Lie algebra $\mathfrak{sl}(V_\C)$ acts naturally on $\wedge^k(V_\C)$ by $$X \cdot (v_1 \wedge \cdots \wedge v_k) = \sum_{i = 1}^k v_1 \wedge \cdots \wedge v_{i-1} \wedge Xv_i \wedge v_{i+1} \wedge \cdots \wedge v_k$$ for all $X \in \mathfrak{sl}(V_\C)$ and $v_i \in V_\C$. Similarly we have an action of $\mathfrak{sl}(V_\C)$ on $\wedge^k(V_\C^*)$. Furthermore, $\mathfrak{sl}(V_\C)$ acts on $\wedge^k(V_\C) \otimes \wedge^k(V_\C^*)$ by $X \cdot (v \otimes w) = Xv \otimes w + v \otimes Xw$ for all $X \in \mathfrak{sl}(V_\C)$, $v \in \wedge^k(V_\C)$ and $w \in \wedge^k(V_\C^*)$. Here $\wedge^k(V_\Z) \otimes \wedge^k(V_\Z^*)$ is an $\mathscr{U}_\Z$-invariant lattice in $\wedge^k(V_\C) \otimes \wedge^k(V_\C^*)$, and we can and will identify $\wedge^k(V_\Z) \otimes \wedge^k(V_\Z^*) \otimes_\Z \fieldsymbol$ and $\wedge^k(V) \otimes \wedge^k(V^*)$ as $G$-modules.

The diagonal matrices in $G$ form a maximal torus $T$. Then a basis of weight vectors of $\wedge^k(V) \otimes \wedge^k(V^*)$ is given by the elements $(e_{i_1} \wedge \cdots \wedge e_{i_k}) \otimes (e_{j_1}^* \wedge \cdots \wedge e_{j_k}^*)$, where $1 \leq i_1 < \cdots < i_k \leq n$ and $1 \leq j_1 < \cdots < j_k \leq n$. The basis vector $(e_1 \wedge \cdots \wedge e_k) \otimes (e_{n}^* \wedge e_{n-1}^* \wedge \cdots \wedge e_{n-k+1}^*)$ has weight $\omega_k + \omega_{n-k}$.

The natural dual pairing between $\wedge^k(V)$ and $\wedge^k(V^*)$ (see for example \cite[B.3, pg.475-476]{FultonHarris}) induces a $G$-invariant symmetric form $\langle -,- \rangle$ on $\wedge^k(V) \otimes \wedge^k(V^*)$. If $x = (v_1 \wedge \cdots \wedge v_k) \otimes (f_1 \wedge \cdots \wedge f_k)$ and $y = (w_1 \wedge \cdots \wedge w_k) \otimes (g_1 \wedge \cdots \wedge g_k)$, with $v_i, w_j \in V$ and $f_i, g_j \in V^*$, we define  $$\langle x, y \rangle = \det (f_i(w_j))_{1 \leq i,j \leq k} \det (g_i(v_j))_{1 \leq i,j \leq k}.$$  Let $b = (e_{i_1} \wedge \cdots \wedge e_{i_k}) \otimes (e_{j_1}^* \wedge \cdots \wedge e_{j_k}^*)$ and $b' = (e_{i_1'} \wedge \cdots \wedge e_{i_k'}) \otimes (e_{j_1'}^* \wedge \cdots \wedge e_{j_k'}^*)$ be two basis elements of $\wedge^k(V) \otimes \wedge^k(V^*)$. Then $$\langle b, b' \rangle = \begin{cases} \pm 1, & \mbox{if } \{i_1, \cdots, i_k\} = \{j_1', \cdots, j_k'\} \mbox{ and } \{i_1', \cdots, i_k'\} = \{j_1, \cdots, j_k\}.  \\ 0, & \mbox{otherwise.} \end{cases} $$

Therefore the form $\langle -,- \rangle$ on $\wedge^k(V) \otimes \wedge^k(V^*)$ is non-degenerate. In precisely the same way we can find a basis of weight vectors for $\wedge^k(V_\Z) \otimes \wedge^k(V_\Z^*)$ and define a symmetric form $\langle -,- \rangle_\Z$ on $\wedge^k(V_\Z) \otimes \wedge^k(V_\Z^*)$. 

We can find the Weyl module $V(\omega_k + \omega_{n-k})$ as a submodule of $\wedge^k(V) \otimes \wedge^k(V^*)$, as shown by the following lemma (cf. Lemma \ref{weyllemma}).

\begin{lemma}\label{weyllemmaA}
Let $1 \leq k < n-k \leq l$, and let $W$ be the $G$-submodule of $\wedge^k(V) \otimes \wedge^k(V^*)$ generated by $v^+ = (e_1 \wedge \cdots \wedge e_k) \otimes (e_{n}^* \wedge e_{n-1}^* \wedge \cdots \wedge e_{n-k+1}^*)$. Then $W$ is isomorphic to the Weyl module $V(\omega_k + \omega_{n-k})$.
\end{lemma}

\begin{proof} It is a general fact about Weyl modules that $V(\lambda) \otimes V(\mu)$ always has $V(\lambda + \mu)$ as a submodule. For simple groups of classical type (in particular, for our $G$ of type $A_l$) this follows from results proven first by Lakshmibai et. al. \cite[Theorem 2 (b)]{LakshmibaiMusiliSeshadri} or from a more general result of Wang \cite[Theorem B, Lemma 3.1]{Wang}. For other types, the fact is a consequence of results due to Donkin \cite{Donkin} (all types except $E_7$ and $E_8$ in characteristic two) or Mathieu \cite{Mathieu} (in general). In any case, now the weight $\lambda + \mu$ occurs with multiplicity $1$ in $V(\lambda) \otimes V(\mu)$, so any vector of weight $\lambda + \mu$ in $V(\lambda) \otimes V(\mu)$ will generate a submodule isomorphic to $V(\lambda + \mu)$.

To prove our lemma, note that $\wedge^k(V) = L(\omega_k)$ and $\wedge^k(V^*) = L(\omega_{n-k})$. Furthermore, $\omega_k$ and $\omega_{n-k}$ are minuscule weights, so $L(\omega_k) = V(\omega_k)$ and $L(\omega_{n-k}) = V(\omega_{n-k})$. Here $v^+$ is a vector of weight $\omega_k + \omega_{n-k}$ in $\wedge^k(V) \otimes \wedge^k(V^*)$, so the claim follows from the result in the previous paragraph.
\end{proof}

For all $1 \leq k < n-k \leq l$, we will identify $V(\omega_k + \omega_{n-k})$ with the submodule $W$ from Lemma \ref{weyllemmaA}. Set $V(\omega_k + \omega_{n-k})_\Z = \mathscr{U}_\Z v^+$ where $v^+$ is as in Lemma \ref{weyllemmaA}. Then we can and will identify $V(\omega_k + \omega_{n-k})_\Z \otimes_\Z \fieldsymbol$ and $V(\omega_k + \omega_{n-k})$ as $G$-modules.

Note that a basis for the zero weight space of $\wedge^k(V) \otimes \wedge^k(V^*)$ is given by vectors of the form $(e_{i_1} \wedge \cdots \wedge e_{i_k}) \otimes (e_{i_1}^* \wedge \cdots \wedge e_{i_k}^*)$, where $1 \leq i_1 < \cdots < i_k \leq n$. We will need the following lemma, which gives a set of generators for the zero weight space of $V(\omega_k + \omega_{n-k})$ (cf. Lemma \ref{weightzerolemma}).

\begin{lemma}\label{weightzerolemmaA}
Suppose that $1 \leq k < n-k \leq l$. Then the zero weight space of $V(\omega_k + \omega_{n-k})_\Z$ (thus also of $V(\omega_k + \omega_{n-k})$) is spanned by vectors of the form 

$$\sum_{f_s \in \{j_s, k_s\}}(-1)^{|\{s : f_s = j_s\}|} (e_{f_1} \wedge \cdots \wedge e_{f_k}) \otimes (e_{f_1}^* \wedge \cdots \wedge e_{f_k}^*),$$ where $(k_1, \ldots, k_k)$ and $(j_1, \ldots, j_k)$ are sequences such that $1 \leq k_r < j_r \leq n$ for all $r$, and $j_r, k_r \neq j_{r'}, k_{r'}$ for all $r \neq r'$.
\end{lemma}

\begin{proof}
We give a proof somewhat similar to that of Lemma \ref{weightzerolemma} given in \cite[pg. 40, Lemma 6]{JantzenThesis}. The zero weight space of $V(\omega_k + \omega_{n-k})_\Z$ is generated by elements of the form $$\prod_{\alpha \in \Phi^+} \frac{X_{-\alpha}^{k_\alpha}}{k_\alpha!}v^+$$ where $k_\alpha$ are non-negative integers, $\sum_{\alpha \in \Phi^+} k_\alpha \alpha = \omega_k + \omega_{n-k}$ and the product is taken with respect to some fixed ordering of the positive roots. For $\alpha \in \Phi^+$ such that $X_{-\alpha}v^+ = 0$, we can assume $k_\alpha = 0$ by choosing a suitable ordering of $\Phi^+$. Therefore we will assume that if $k_\alpha > 0$, then $\alpha$ is of one of the following types.

\begin{enumerate}[(I)]
\item $\alpha = \varepsilon_i - \varepsilon_j$ for $1 \leq i \leq k$ and $k+1 \leq j \leq n-k$.
\item $\alpha = \varepsilon_i - \varepsilon_j$ for $k+1 \leq i \leq n-k$ and $n-k+1 \leq j \leq n$.
\item $\alpha = \varepsilon_i - \varepsilon_j$ for $1 \leq i \leq k$ and $n-k+1 \leq j \leq n$.
\end{enumerate}

Note that the $X_{-\alpha}$ with $\alpha$ of type (I) commute with each other. The same is also true for types (II) and (III). 

Writing $\omega_k + \omega_{n-k}$ in terms of the simple roots, we see that $\omega_k + \omega_{n-k}$ is equal to 
\begin{equation}
\alpha_1 + 2 \alpha_2 + \cdots + (k-1) \alpha_{k-1} + k \alpha_k + \cdots + k \alpha_{n-k} + (k-1) \alpha_{n-k+1} + \cdots + \alpha_{n-1} \tag{*}
\end{equation}
Then from the fact that $\sum_{\alpha \in \Phi^+} k_\alpha \alpha = \omega_k + \omega_{n-k}$ we will deduce the following. 

\begin{enumerate}[(1)]
\item For any $1 \leq i \leq k$, there exists a unique $\alpha \in \Phi^+$ such that $k_{\alpha} = 1$ and $\alpha = \varepsilon_i - \varepsilon_{j'}$ for some $k+1 \leq j' \leq n$. 
\item For any $n-k+1 \leq j \leq n$ there exists a unique $\alpha \in \Phi^+$ such that $k_{\alpha} = 1$ and $\alpha = \varepsilon_{i'} - \varepsilon_j$ for some $1 \leq i' \leq n-k$. 
\end{enumerate}

For $i = 1$ and $j = n$ these claims are clear, since $\alpha_1$ and $\alpha_{n-1}$ occur only once in the expression (*) of $\omega_k + \omega_{n-k}$ as a sum of simple roots. For $i > 1$ claim (1) follows by induction, since $\varepsilon_i - \varepsilon_{j'}$ contributes $\alpha_i + \alpha_{i+1} + \cdots + \alpha_{k} + \cdots + \alpha_{j'-1}$ to the expression (*) of $\omega_k + \omega_{n-k}$ as a sum of simple roots. Claim (2) follows similarly for $j < n$.

In particular, it follows from claims (1) and (2) that $k_\alpha \in \{0, 1\}$ for all $\alpha \in \Phi^+$. Let $\mathscr{A}_1$, $\mathscr{A}_2$ and $\mathscr{A}_3$ be the sets of $\alpha \in \Phi^+$ of type (I), (II) and (III) respectively such that $k_{\alpha} = 1$. It follows from claim (1) that $|\mathscr{A}_1| + |\mathscr{A}_3| = k$ and from claim (2) that $|\mathscr{A}_2| + |\mathscr{A}_3| = k$, so then $|\mathscr{A}_1| = |\mathscr{A}_2 |= k'$ for some $0 \leq k' \leq k$. Thus we can write \begin{align*}
\mathscr{A}_1 &= \{ \varepsilon_{i_1} - \varepsilon_{w_1}, \ldots, \varepsilon_{i_{k'}} - \varepsilon_{w_{k'}} \} \\ 
\mathscr{A}_2 &= \{ \varepsilon_{z_1} - \varepsilon_{j_1}, \ldots, \varepsilon_{z_{k'}} - \varepsilon_{j_{k'}} \} \\ 
\mathscr{A}_3 &= \{ \varepsilon_{i_{k'+1}} - \varepsilon_{j_{k'+1}}, \ldots, \varepsilon_{i_k} - \varepsilon_{j_k}\}
\end{align*}
where $1 \leq i_r \leq k$ and $n-k+1 \leq j_r \leq n$ for all $1 \leq r \leq k$, and $k+1 \leq w_r, z_r \leq n-k$ for all $1 \leq r \leq k'$. Furthermore, $\{i_1, \ldots, i_k\} = \{1, 2, \ldots, k\}$ and $\{j_1, \ldots, j_k\} = \{n-k+1, \ldots, n-1, n\}$.

We choose the ordering of $\Phi^+$ so that $$\prod_{\alpha \in \Phi^+} \frac{X_{-\alpha}^{k_\alpha}}{k_\alpha!} = \prod_{\alpha \in \mathscr{A}_3} X_{-\alpha} \prod_{\alpha \in \mathscr{A}_2} X_{-\alpha}  \prod_{\alpha \in \mathscr{A}_1} X_{-\alpha}.$$

It is another consequence of $\sum_{\alpha \in \Phi^+} k_\alpha \alpha = \omega_k + \omega_{n-k}$ that $\{w_1, \ldots, w_{k'}\} = \{z_1, \ldots, z_{k'}\}$. Indeed, in the expression (*) of $\omega_k + \omega_{n-k}$ as a sum of simple roots, for any $k+1 \leq r \leq n-k$ the simple root $\alpha_r$ occurs $k$ times. On the other hand, the $\alpha$ of types (I), (II), (III) that contribute to $\alpha_r$ in the sum are precisely those of type (I) or (III) with $j > r$ (total of $k - |\{r' : w_{r'} \leq r\}|$), and those of type (II) with $j \leq r$ (total of $|\{r' : z_{r'} \leq r \}|$). 

Therefore in the sum $\sum_{\alpha \in \Phi^+} k_\alpha \alpha$, the contribution to $\alpha_r$ is equal to $k - |\{r' : w_{r'} \leq r\}| + |\{r' : z_{r'} \leq r \}|$. Since this has to be equal to $k$, we get $|\{r' : z_{r'} \leq r \}| = |\{r' : w_{r'} \leq r\}|$ for all $k+1 \leq r \leq n-k$, which implies $\{w_1, \ldots, w_{k'}\} = \{z_1, \ldots, z_{k'}\}$.

Then since the $X_{-\alpha}$ with $\alpha$ of type (II) commute with each other, we may assume that $z_r = w_r$ for all $1 \leq r \leq k'$. Denote $w = \prod_{r = 1}^{k'} E_{w_r, i_r} v^+$. A straightforward computation shows that $w = (e_{\pi(1)} \wedge \cdots \wedge e_{\pi(k)}) \otimes (e_{n}^* \wedge e_{n-1}^* \wedge \cdots \wedge e_{n-k+1}^*)$, where $\pi(r) = w_{r'}$ if $r = i_{r'}$ and $\pi(r) = r$ otherwise. Now \begin{align*}
\prod_{\alpha \in \Phi^+} \frac{X_{-\alpha}^{k_\alpha}}{k_\alpha!}v^+ &= \prod_{r={k'+1}}^{k} E_{j_r, i_r} \prod_{r = 1}^{k'} E_{j_r, w_r} \prod_{r = 1}^{k'} E_{w_r, i_r} v^+\\
&= \prod_{r={k'+1}}^{k} E_{j_r, i_r} \prod_{r = 1}^{k'} E_{j_r, w_r} w \\
&= \prod_{r={1}}^{k} E_{j_r, k_r} w
\end{align*} where $(k_1, \ldots, k_k) = (w_1, \ldots, w_{k'}, i_{{k'}+1}, \ldots, i_k)$. In the last equality we just combine the terms, and this makes sense since $X_{-\alpha}$ of type (III) commute with those of type (II).



Computing the expression $\prod_{r={1}}^{k} E_{j_r, k_r} w$, we see that it is equal to a sum of $2^k$ distinct elements of $\wedge^k(V_\Z) \otimes \wedge^k(V_\Z^*)$, with each summand being equal to $w$ transformed in the following way:

\begin{itemize}
\item For all $1 \leq s \leq k'$, replace $e_{w_s}$ by $e_{j_s}$, or replace $e_{j_s}^*$ by $-e_{w_s}^*$.
\item For all $k'+1 \leq s \leq k$, replace $e_{i_s}$ by $e_{j_s}$, or replace $e_{j_s}^*$ by $-e_{i_s}^*$.
\end{itemize}

For this we conclude that up to a sign, $\prod_{\alpha \in \Phi^+} \frac{X_{-\alpha}^{k_\alpha}}{k_\alpha!}v^+$ is as in the statement of the lemma, with sequences $(k_1, \ldots, k_k)$ and $(j_1, \ldots, j_k)$ as defined here.\end{proof}

\begin{lemma}\label{wedgefixedpointA}
Suppose that $1 \leq k < n-k \leq l$. Then the vector $$\sum_{1 \leq i_1 < \cdots < i_{k} \leq n} (e_{i_1} \wedge \cdots \wedge e_{i_{k}}) \otimes (e_{i_1}^* \wedge \cdots \wedge e_{i_{k}}^*)$$ is fixed by the action of $G$ on $\wedge^k(V) \otimes \wedge^k(V^*)$. Furthermore, any $G$-fixed point in $\wedge^k(V) \otimes \wedge^k(V^*)$ is a scalar multiple of $\gamma$.
\end{lemma}

\begin{proof}
(cf. Lemma \ref{wedgefixedpoint})  The fact that $\gamma$ is fixed by $G$ is an exercise in linear algebra. We will give a proof for convenience of the reader. For this we first need to introduce some notation. Let $A$ be an $n \times n$ matrix with entries in $\fieldsymbol$ and denote the entry on $i$th row and $j$th column of $A$ by $A_{i,j}$. For indices $1 \leq i_1 < \cdots < i_k \leq n$ and $1 \leq j_1 < \cdots < j_k \leq n$, set $I = \{i_1, \ldots, i_k\}$ and $J = \{j_1, \ldots, j_k\}$. Then the $k \times k$ minor of $A$ defined by $I$ and $J$ is the determinant of the $k \times k$ matrix $(A_{i_p, j_q})$. We denote this minor by $[A]_{I,J}$. The following relation between minors of matrices $A$, $B$, and $AB$ (similar to the matrix multiplication rule) is a special case of the Cauchy-Binet formula. A proof can be found in \cite[4.6, pg. 208-214]{BroidaWilliamson}.

\begin{prop*}[Cauchy-Binet formula]
Let $A$ and $B$ be $n \times n$ matrices. For any $k$-element subsets $I, J$ of $\{1, \ldots, n\}$, we have $$[AB]_{I, J} = \sum_{T} [A]_{I, T} [B]_{T, J}$$ where the sum runs over all $k$-element subsets $T$ of $\{1, \ldots, n\}$.
\end{prop*}

Consider $A \in \GL(V)$ as a matrix with respect to the basis $e_1, \ldots, e_n$ of $V$. Now for any $1 \leq k \leq n$, the matrix $A$ acts on the exterior power $\wedge^k(V)$ by $A \cdot (v_1 \wedge \cdots \wedge v_k) = Av_1 \wedge \cdots \wedge Av_k$ for all $v_i \in V$. For $I = \{i_1, \ldots, i_k\}$ with $1 \leq i_1 < \cdots < i_k \leq n$, denote $e_I = e_{i_1} \wedge \cdots \wedge e_{i_k}$. A straightforward calculation shows that $A \cdot e_I = \sum_{J} [A]_{J, I} e_J$, where the sum runs over all $k$-element subsets $J$ of $\{1, \ldots, n\}$. 

With respect to the dual basis $e_1^*, \ldots, e_n^*$ of $V^*$, the action of $A$ on $V^*$ has matrix $(A^t)^{-1}$, where $A^t$ is the transpose of $A$. If we denote $e_I^* = e_{i_1}^* \wedge \cdots \wedge e_{i_k}^*$ as before, then $A \cdot e_I^* = \sum_{J} [(A^t)^{-1}]_{J, I} e_J^* = \sum_{J} [A^{-1}]_{I, J} e_J^*$ where the sum runs over all $k$-element subsets $J$ of $\{1, \ldots, n\}$.

We are now ready to prove that $A$ fixes the vector $\gamma$. Note that $\gamma = \sum_{I} e_I \otimes e_I^*$, where the sum runs over all $k$-element subsets $I$ of $\{1, \ldots, n\}$. From the observations before, we see that $A \cdot \gamma$ is equal to $$\sum_{I} \sum_{J,J'} [A]_{J, I} [A^{-1}]_{I, J'} e_J \otimes e_{J'}^* = \sum_{J,J'} \sum_{I} [A]_{J, I} [A^{-1}]_{I, J'} e_J \otimes e_{J'}^*$$ where the sums run over $k$-element subsets $I$, $J$, and $J'$ of $\{1, \ldots, n\}$. From the Cauchy-Binet formula, we have $\sum_{I} [A]_{J, I} [A^{-1}]_{I, J'} = [1]_{J, J'} = 0$ if $J \neq J'$ and $1$ if $J = J'$. Therefore $A \cdot \gamma = \gamma$.

To show that $\gamma$ is a unique $G$-fixed point up to a scalar, note first that any $G$-fixed point must have weight zero. Recall also that the zero weight space of $\wedge^k(V)$ has basis $$\mathscr{B} = \{(e_{i_1} \wedge \cdots \wedge e_{i_{k}}) \otimes (e_{i_1}^* \wedge \cdots \wedge e_{i_{k}}^*) : 1 \leq i_1 < \cdots < i_k \leq n \}.$$ 

Now the group $\Sigma_{n}$ of permutations of $\{1,2, \ldots, n\}$ acts on $V$ by $\sigma \cdot e_{i} = e_{\sigma(i)}$ for all $\sigma \in \Sigma_{n}$. This gives an embedding $\Sigma_{n} < \GL(V)$. Note that then $\sigma \cdot e_i^* = e_{\sigma(i)}^*$ for all $\sigma \in \Sigma_{n}$, so it follows that $\Sigma_{n}$ acts on $\mathscr{B}$.

For $\sigma \in \Sigma_{n}$ we have $\det \sigma = 1$ if and only if $\sigma$ is an even permutation, so we get an embedding $\operatorname{Alt}(n) < G$ for the alternating group. It is well known that $\operatorname{Alt}(n)$ is $(n-2)$-transitive, so $\operatorname{Alt}(n)$ acts transitively on $\mathscr{B}$ since $k \leq n-2$. Thus any $\operatorname{Alt}(n)$-fixed point in the linear span of $\mathscr{B}$ must be a scalar multiple of $\sum_{b \in \mathscr{B}} b = \gamma$.
\end{proof}

We now begin the proof of Theorem \ref{quadpropA}. For the rest of this section, we will make the following assumption. 

\begin{center}\emph{Assume that $\Ch \fieldsymbol = 2$.}\end{center}

Let $1 \leq k < n-k \leq l$ and suppose that $L(\omega_k + \omega_{n-k})$ is not orthogonal. Then by Proposition \ref{char2quad} (iii) we have $\operatorname{H}^1(G, L(\omega_k + \omega_{n-k})) \neq 0$, so by Corollary \ref{extcorollaryA} we have $k = 2^{i}$ for some $i \geq 0$ and $n+1 \equiv 2^i + t \mod{2^{i+1}}$ for some $0 \leq t < 2^i$. What remains is to determine when $L(\omega_k + \omega_{n-k})$ is orthogonal for such $k$. The main argument is the following lemma (cf. Lemma \ref{mainlemma}), which reduces the question to the evaluation of the invariant quadratic form on $V(\omega_k + \omega_{n-k})$ on a particular $v \in V(\omega_k + \omega_{n-k})$.

\begin{lemma}\label{mainlemmaA}
Let $n > 2^{i+1}$, where $i \geq 0$. Suppose that $n + 1 \equiv 2^i + t \mod{2^{i+1}}$, where $0 \leq t < 2^i$. Define the vector $\gamma \in \wedge^{2^{i}}(V) \otimes \wedge^{2^i}(V^*)$ to be equal to $$\sum_{1 \leq i_1 < \cdots < i_{2^i} \leq n-t} (e_{i_1} \wedge \cdots \wedge e_{i_{2^i}}) \otimes (e_{i_1}^* \wedge \cdots \wedge e_{i_{2^i}}^*).$$ Then 

\begin{enumerate}[\normalfont (i)]
\item $\gamma$ is in $V(\omega_{2^{i}} + \omega_{n-2^i})$ and is a fixed point for the subgroup $A_{l-t} < G$,
\item $\gamma$ is in $\rad V(\omega_{2^{i}} + \omega_{n-2^i})$,
\item $L(\omega_{2^{i}} + \omega_{n-2^i})$ is orthogonal if and only if $Q(\gamma) = 0$, where $Q$ is a nonzero $G$-invariant quadratic form on $V(\omega_{2^{i}} + \omega_{n-2^i})$.
\end{enumerate}

\end{lemma}

\begin{proof}

\begin{enumerate}[(i)]

\item Same as Lemma \ref{mainlemma} (i). Apply Lemma \ref{wedgefixedpointA}, Lemma \ref{weyllemmaA} and Corollary \ref{trivialsubmoduleA}.


\item Same as Lemma \ref{mainlemma} (ii). Apply Lemma \ref{weightzerolemmaA}, and note that for $b = (e_{j_1} \wedge \cdots \wedge e_{j_{2^i}}) \otimes (e_{j_1}^* \wedge \cdots \wedge e_{j_{2^i}}^*)$ and $b' = (e_{k_1} \wedge \cdots \wedge e_{k_{2^i}}) \otimes (e_{k_1}^* \wedge \cdots \wedge e_{k_{2^i}}^*)$ we have $$\langle b, b' \rangle = \begin{cases} 1, & \mbox{if } \{j_1, \ldots, j_{2^i}\} = \{k_1, \ldots, k_{2^i}\}. \\ 0, & \mbox{otherwise.} \end{cases}$$

\item Same as Lemma \ref{mainlemma} (iii). Apply Theorem \ref{extcorollaryA} to find a submodule $M \subseteq \rad V(\omega_{2^i} + \omega_{n-2^i})$ such that $\rad V(\omega_{2^{i+1}}) / M \cong \fieldsymbol$. Each composition factor of $V(\omega_{2^i} + \omega_{n-2^i})$ occurs with multiplicity one by Theorem \ref{typeAweylcf}, so $M$ has no nontrivial composition factors. By Lemma \ref{restrictionlemmaA} (i) the restriction $M \downarrow A_{l-t}$ has no trivial composition factors, and by (i) the vector $\gamma$ is fixed by $A_{l-t}$. Thus $\gamma \not\in M$ and then $\rad V(\omega_{2^{i}} + \omega_{n-2^i}) = \langle \gamma \rangle \oplus M$ as $A_{l-t}$-modules. As in Lemma \ref{mainlemma} (iii), we see that $L(\omega_{2^i} + \omega_{n-2^i})$ is orthogonal if and only if $Q(\gamma) = 0$ for a nonzero $G$-invariant quadratic form $Q$ on $V(\omega_{2^{i}} + \omega_{n-2^i})$.

\end{enumerate}
\end{proof}

\subsection{Computation of a quadratic form $Q$ on $V(\omega_r + \omega_{n-r})$}

\noindent To finish the proof of Theorem \ref{quadpropA} we still have to compute $Q(\gamma)$ for the vector $\gamma$ from Lemma \ref{mainlemmaA}.

We retain the notation and assumptions from the previous subsection. Let $1 \leq k < n-k \leq l$. Now the form $\langle -,- \rangle_\Z$ on $\wedge^k(V_\Z) \otimes \wedge^k(V_\Z^*)$ induces a quadratic form $q_\Z$ on $\wedge^k(V_\Z) \otimes \wedge^k(V_\Z^*)$ by $q_\Z(x) = \langle x, x \rangle$. We will use this form to find a nonzero $G$-invariant quadratic form on $V(\omega_k + \omega_{n-k}) = V(\omega_k + \omega_{n-k})_\Z \otimes_\Z \fieldsymbol$.

\begin{lemma}
We have $q_\Z(V(\omega_k + \omega_{n-k})_\Z) \subseteq 2 \Z$ and $q_\Z(V(\omega_k + \omega_{n-k})_\Z) \not\subseteq 4 \Z$.
\end{lemma}

\begin{proof}Same as Lemma \ref{quadevenlemma}, but with $\alpha = (e_1 \wedge \cdots \wedge e_k) \otimes (e_{n}^* \wedge e_{n-1}^* \wedge \cdots \wedge e_{n-k+1}^*)$ and $\beta = (e_{n} \wedge e_{n-1} \wedge \cdots \wedge e_{n-k+1}) \otimes (e_1^* \wedge \cdots \wedge e_k^*)$.\end{proof}

Therefore $Q = \frac{1}{2} q_\Z \otimes_\Z \fieldsymbol$ defines a nonzero $G$-invariant quadratic form on $V(\omega_k + \omega_{n-k})$ with polarization $\langle -,- \rangle$. As in Section \ref{typeCquadcomp}, we have $Q(\gamma) = \frac{1}{2} \binom{n-t}{2^i}$ for the vector $\gamma \in \wedge^{2^i}(V) \otimes \wedge^{2^i}(V^*)$ from Lemma \ref{mainlemmaA}. Finally applying Lemma \ref{binomiallemma} completes the proof Theorem \ref{quadpropA}.

\section{Simple groups of exceptional type}\label{exceptionalsection}

\noindent In this section, let $G$ be a simple group of exceptional type and assume that $\Ch \fieldsymbol  = 2$. We will give some results about the orthogonality of irreducible representations of $G$. For $G$ of type $G_2$ or $F_4$ we give a complete answer. For types $E_6$, $E_7$, and $E_8$ we only have results for some specific representations, given in Table \ref{table:typeE} below and proven at the end of this section. For irreducible representations occurring in the adjoint representation of $G$, answers were given earlier by Gow and Willems in \cite[Section 3]{GowWillems1}.

\begin{prop}\label{prop:g2}
Let $G = G_2$ and let $V$ be a non-trivial irreducible representation of $G$. Then $V$ is not orthogonal if and only if $V$ is a Frobenius twist of $L_G(\omega_1)$.
\end{prop}

\begin{proof}
In view of Remark \ref{restrictedremark}, it will be enough to consider $V = L_G(\lambda)$ with $\lambda \in X(T)^+$ a $2$-restricted dominant weight. If $\lambda = \omega_2$ or $\lambda = \omega_1 + \omega_2$, then $V_G(\lambda) = L_G(\lambda)$ and so $V$ is orthogonal by Proposition \ref{char2quad}. 

What remains is to show that $V = L_G(\omega_1)$ not orthogonal. There are several ways to see this, for example since $\dim V = 6$ this could be done by a direct computation. Alternatively, note that the composition factors of $\wedge^2(V)$ are $L_G(\omega_2)$ and $L_G(0)$ \cite[Proposition 2.10]{LiebeckSeitzReductive}, so $\operatorname{H}^1(G, \wedge^2(V)) = 0$. Then by \cite[Proposition 2.7]{SinWillems}, the module $V$ is not orthogonal. For a third proof, note that the action of a regular unipotent $u \in G$ on $V$ has a single Jordan block \cite[Theorem 1.9]{Suprunenko}, but no such element exists in $\SO(V)$ \cite[Proposition 6.22]{LiebeckSeitzClass}.\end{proof}

The following lemma will be useful throughout this section to show that certain representations are orthogonal.

\begin{lemma}\label{lemma:altsquarecombined}
Let $V$ be a nontrivial, self-dual and irreducible $G$-module. Suppose that one of the following holds:

\begin{enumerate}[\normalfont (i)]
\item $\dim V \equiv 2 \mod{4}$, and $\wedge^2(V)$ has exactly one trivial composition factor as a $G$-module.
\item $\dim V \equiv 0 \mod{8}$, and $\wedge^2(V)$ has exactly two trivial composition factors as a $G$-module.
\end{enumerate}

Then any nontrivial composition factor of $\wedge^2(V)$ occuring with odd multiplicity is an orthogonal $G$-module.
\end{lemma}

\begin{proof}
Since $V$ is nontrivial, we can assume $G < \Sp(V)$ by Lemma \ref{fonglemma}. If (i) holds, then by applying results in Section \ref{weylconstructionC} (or \cite[Lemma 4.8.2]{McNinch}) we can find a vector $\gamma \in \wedge^2(V)$ such that $\wedge^2(V) = Z \oplus \langle \gamma \rangle$ as an $\Sp(V)$-module. Here $Z$ is irreducible of highest weight $\omega_2$ for $\Sp(V)$, so by Proposition \ref{fundamentalquadC} (see Example \ref{exampleC2}) the module $Z$ is orthogonal for $\Sp(V)$. Therefore $Z$ is an orthogonal $G$-module with no trivial composition factors. From this \cite[Lemma 1.3]{GowWillems1} shows that any composition factor of $Z$ with odd multiplicity is an orthogonal $G$-module. 

In case (ii), the assumption on $\dim V$ implies (for example by \cite[Lemma 4.8.2]{McNinch}) that there exist $\Sp(V)$-submodules $Z' \subseteq Z \subseteq \wedge^2(V)$ such that $\dim Z' = \dim \wedge^2(V) / Z = 1$. Furthermore, $Z/Z'$ is an irreducible $\Sp(V)$-module with highest weight $\omega_2$, so by Proposition \ref{fundamentalquadC} (see Example \ref{exampleC2}) the module $Z/Z'$ is orthogonal for $\Sp(V)$. Therefore $Z/Z'$ is an orthogonal $G$-module with no trivial composition factors, so by \cite[Lemma 1.3]{GowWillems1} any composition factor of $Z/Z'$ with odd multiplicity is an orthogonal $G$-module.\end{proof}

\begin{prop}\label{F4result}
Let $G = F_4$ and let $V$ be a non-trivial irreducible representation of $G$. Then $V$ is orthogonal.
\end{prop}

\begin{proof}
Let $\tau: G \rightarrow G$ be the exceptional isogeny of $G$ as given in \cite[Theorem 28]{SteinbergNotes}. Then $L_G(a_1\omega_1 + a_2 \omega_2 + a_3 \omega_3 + a_4 \omega_4)^\tau \cong L_G(a_4 \omega_1 + a_3 \omega_2 + 2 a_2 \omega_3 + 2 a_1 \omega_4)$, and by Steinberg's tensor product theorem this is isomorphic to $L_G(a_4 \omega_1 + a_3 \omega_2) \otimes L(a_2 \omega_3 + a_1 \omega_4)^{F}$. Thus by lemmas \ref{fonglemma} and \ref{tensorlemma}, it is enough to prove the claim in the case where $V = L_G(\lambda)$ with $\lambda = a_3 \omega_3 + a_4 \omega_4$ a $2$-restricted dominant weight. Now for $\lambda = \omega_4$ and $\lambda = \omega_3 + \omega_4$, we have $V_G(\lambda) = L_G(\lambda) = V$ and so $V$ is orthogonal by Proposition \ref{char2quad} (iv)\footnote{One can also construct a non-degenerate $G$-invariant quadratic form on $L_G(\omega_4)$ explicitly by realizing it as the space of trace zero elements in the Albert algebra. The details of this construction can be found in \cite[4.8.4, pg. 151-152]{Wilson}.}. 

What remains is to show that $L_G(\omega_3)$ is orthogonal. Let $W = L_G(\omega_4)$. Now $\dim W = 26$, so by Lemma \ref{lemma:altsquarecombined} (i), it will be enough to prove that $\wedge^2(W)$ has exactly one trivial composition factor and that $L_G(\omega_3)$ occurs in $\wedge^2(W)$ with odd multiplicity.

We have $V_G(\omega_4) = L_G(\omega_4)$ and then by a computation with Magma \cite{MAGMA} (or \cite[7.4.3, pg. 98]{Donkin}) the $G$-character of $\wedge^2(W)$ is given by $\ch \wedge^2(W) = \ch V_G(\omega_1) + \ch V_G(\omega_3)$. Furthermore, from the data in \cite{LubeckWebsite}, we can deduce $V_G(\omega_1) = L_G(\omega_1) / L_G(\omega_4)$ and $V_G(\omega_3) = L_G(\omega_3) / L_G(\omega_4) / L_G(0)$. Therefore as a $G$-module $\wedge^2(W)$ has composition factors $L_G(\omega_1)$, $L_G(\omega_4)$, $L_G(\omega_4)$, $L_G(\omega_3)$, and $L_G(0)$.\end{proof}

\begin{table}[htp]
\centering

\begin{tabular}{| l | l | l |}
\hline
$G$ & $\lambda$ & $L_G(\lambda)$ orthogonal? \\
\hline
$E_6$ & $\omega_2$ & yes \\
$E_6$ & $\omega_4$ & yes \\
$E_6$ & $\omega_1 + \omega_6$ & yes \\

$E_7$ & $\omega_1$ & no \\
$E_7$ & $\omega_2$ & yes \\
$E_7$ & $\omega_5$ & yes \\
$E_7$ & $\omega_6$ & yes \\
$E_7$ & $\omega_7$ & yes \\

$E_8$ & $\omega_1$ & yes \\
$E_8$ & $\omega_7$ & yes \\
$E_8$ & $\omega_8$ & yes \\

\hline
\end{tabular}
\caption{Orthogonality of some $L_G(\lambda)$ for $G$ of type $E_6$, $E_7$ and $E_8$.}\label{table:typeE}
\end{table}

We finish this section by verifying the information given in Table \ref{table:typeE}. 

Suppose that $G$ is of type $E_6$. We have $V_G(\omega_2) = L_G(\omega_2)$ and $V_G(\omega_1 + \omega_6) = L(\omega_1 + \omega_6) / L(\omega_2)$ by \cite{LubeckWebsite}, so $L_G(\omega_2)$ and $L_G(\omega_1 + \omega_6)$ are orthogonal by Proposition \ref{char2quad} (iv). We show next that $L(\omega_4)$ is orthogonal. Now $W = L_G(\omega_2)$ is self-dual and $\dim W = 78$, so by Lemma \ref{lemma:altsquarecombined} (ii) it will be enough to prove that $\wedge^2(W)$ has exactly one trivial composition factor and that $L_G(\omega_4)$ occurs in $\wedge^2(W)$ with odd multiplicity.

Now $V_G(\omega_2) = L_G(\omega_2)$, and then a computation with Magma \cite{MAGMA} (or \cite[8.12, pg. 136]{Donkin}) shows that $\ch \wedge^2(W) = \ch V_G(\omega_2) + \ch V_G(\omega_4)$. From \cite{LubeckWebsite}, we can deduce that the composition factors of $V_G(\omega_4)$ are $L_G(\omega_4)$, $L_G(\omega_1 + \omega_6)$, $L_G(\omega_1 + \omega_6)$, $L_G(\omega_2)$ and $L_G(0)$. Thus $L_G(0)$ and $L_G(\omega_4)$ both occur exactly once as a composition factor of $\wedge^2(W)$.

Consider next $G$ of type $E_7$. We can assume that $G$ is simply connected. Then the Weyl module $V_G(\omega_1)$ is the Lie algebra of $G$, and $L_G(\omega_1)$ is not orthogonal by \cite[Theorem 3.4 (a)]{GowWillems1}. We have $V_G(\omega_2) = L_G(\omega_2)$, $V_G(\omega_5) = L_G(\omega_5) / L_G(\omega_1 + \omega_7)$ and $V_G(\omega_7) = L_G(\omega_7)$ by the data in \cite{LubeckWebsite}. Therefore $L_G(\omega_2)$, $L_G(\omega_5)$ and $L_G(\omega_7)$ are orthogonal by Proposition \ref{char2quad} (iv).

We show that $L_G(\omega_6)$ is orthogonal. Now for $W = L_G(\omega_7)$ we have $\dim W = 56$, so by Lemma \ref{lemma:altsquarecombined} (ii), it will be enough to prove that $\wedge^2(W)$ has exactly two trivial composition factors and that $L_G(\omega_6)$ occurs in $\wedge^2(W)$ with odd multiplicity. Now $V_G(\omega_7) = L_G(\omega_7)$, so by a computation with Magma \cite{MAGMA} we see $\ch \wedge^2(W) = \ch V_G(\omega_6) + \ch V_G(0)$. From \cite{LubeckWebsite}, we see that $V_G(\omega_6)$ has composition factors $L_G(\omega_6)$, $L_G(\omega_1)$, $L_G(\omega_1)$ and $L_G(0)$. Therefore $\wedge^2(W)$ has exactly two trivial composition factors and $L_G(\omega_6)$ occurs exactly once as a composition factor.

For $G$ of type $E_8$, we have $V_G(\omega_8) = L_G(\omega_8)$ and so $L_G(\omega_8)$ is orthogonal by \ref{char2quad} (iv). Finally, we show that $L_G(\omega_1)$ and $L_G(\omega_7)$ are orthogonal. For $W = L_G(\omega_8)$ we have $\dim W = 248$, so by Lemma \ref{lemma:altsquarecombined} (ii), it will be enough to prove that $\wedge^2(W)$ has exactly two trivial composition factors and that $L_G(\omega_1)$ and $L_G(\omega_7)$ occur in $\wedge^2(W)$ with odd multiplicity. By a computation with Magma \cite{MAGMA} we see that $\ch \wedge^2(W) = \ch V_G(\omega_8) + \ch V_G(\omega_7)$. From \cite{LubeckWebsite}, we see that $V_G(\omega_7)$ has composition factors $L_G(\omega_7)$, $L_G(\omega_1)$, $L_G(\omega_8)$, $L_G(0)$, and $L_G(0)$. Therefore $\wedge^2(W)$ has exactly two trivial composition factors and both $L_G(\omega_1)$ and $L_G(\omega_7)$ occur with multiplicity one.

\section{Applications and further work}\label{section:final}

\noindent In this section, we describe consequences of some of our findings and propose some questions motivated by Problem \ref{mainproblem}. Unless otherwise mentioned, we let $G$ be a simply connected algebraic group over $\fieldsymbol$ and we assume that $\Ch \fieldsymbol = 2$.

\subsection{Connection with representations of the symmetric group}

\noindent Denote the symmetric group on $n$ letters by $\Sigma_n$. We will describe a connection between orthogonality of certain irreducible $\fieldsymbol [\Sigma_n]$-representations and the irreducible representations $L(\omega_r)$ of $\Sp_{2l}(\fieldsymbol)$. This is done by an application of Proposition \ref{fundamentalquadC} and various results from the literature. The result is not too surprising, since the representation theory of the symmetric group plays a key role in the representation theory of the modules $L(\omega_r)$ of $\Sp_{2l}(\fieldsymbol)$. For example, many of the results that we applied in the proof of Proposition \ref{fundamentalquadC} above are based on studying certain $\fieldsymbol [\Sigma_n]$-representations associated with $V(\omega_r)$.

It is well known that there exists an embedding $\Sigma_{2l+1} < \Sp_{2l}(\fieldsymbol) = G$ for all $l \geq 2$ (see e.g. \cite{GowKleshchev} or \cite[Theorem 8.9]{Taylor}). Therefore if a representation $V$ of $G$ is orthogonal, it is clear that the same is true for the restriction $V \downarrow \Sigma_{2l+1}$. We will proceed to show that the converse is also true when $V = L(\omega_r)$ for $2 \leq r \leq l$, which does not seem to be a priori obvious.

First of all, the following result due to Gow and Kleshchev \cite[Theorem 1.11]{GowKleshchev} gives the structure of $L(\omega_r) \downarrow \Sigma_{2l+1}$.

\begin{lause}\label{gowkleshchevthm}
Let $1 \leq r \leq l$. Then the restriction $L(\omega_r) \downarrow \Sigma_{2l+1}$ is irreducible, and it is isomorphic to the irreducible $\fieldsymbol [\Sigma_n]$-module $D^{(2l+1-r,r)}$ labeled by the partition $(2l+1-r,r)$ of $2l+1$.
\end{lause}

\noindent Now Gow and Quill have determined in \cite{GowQuill} when the irreducible $\fieldsymbol [\Sigma_n]$-modules $D^{(n-r,r)}$ are orthogonal. Their result is the following.

\begin{lause}\label{gowquillthm}
Let $0 \leq r \leq n$. Then the $\fieldsymbol [\Sigma_n]$-module $D^{(n-r,r)}$ is not orthogonal if and only if $r = 2^{j}$, $j \geq 0$ and $n \equiv k \mod{2^{j+2}}$ for some $2^{j+1} + 2^j - 1 \leq k \leq 2^{j+2} - 2$.
\end{lause}

\noindent In the case where $n = 2l+1$, one can express the result in the following way.

\begin{seur}\label{symcorollary}
Let $n = 2l+1$ and $1 \leq r \leq l$. Then $\fieldsymbol [\Sigma_n]$-module $D^{(n-r,r)}$ is not orthogonal if and only if $r = 2^{i+1}$, $i \geq 0$ and $l+1 \equiv 2^{i+1} + 2^i + t \mod{2^{i+2}}$ for some $0 \leq t < 2^i$.
\end{seur}

\begin{proof}
By Theorem \ref{gowquillthm}, the module $D^{(n-1,1)}$ is not orthogonal if and only if $2l+1 \equiv 2 \mod{4}$, which never happens. Therefore $D^{(n-1,1)}$ is always orthogonal, as desired.

Consider then $r > 1$. According to Theorem \ref{gowquillthm}, if $D^{(n-r,r)}$ is not orthogonal, then $r = 2^j$ for some $j > 0$. In this case $D^{(n-r,r)}$ is not orthogonal if and only if $2l+1 \equiv k \mod{2^{j+2}}$ for some $2^{j+1} + 2^j - 1 \leq k \leq 2^{j+2} - 2$. This is equivalent to saying that $2(l+1) \equiv k \mod{2^{j+2}}$ for some $2^{j+1} + 2^j \leq k \leq 2^{j+2} - 1$. Now this condition is equivalent to $l+1 \equiv k \mod{2^{j+1}}$ for some $2^j + 2^{j-1} \leq k \leq 2^{j+1} - 2$, giving the claim.\end{proof}

\noindent Finally combining Theorem \ref{gowkleshchevthm}, Proposition \ref{fundamentalquadC} and Corollary \ref{symcorollary} gives the following result.

\begin{prop}
Let $2 \leq r \leq l$. Then $L(\omega_r)$ is orthogonal for $\Sp_{2l}(k)$ if and only if $L(\omega_r) \downarrow \Sigma_{2l+1}$ is orthogonal for $\Sigma_{2l+1}$.
\end{prop}

\subsection{Reduction for Problem \ref{mainproblem}}

\noindent To determine which irreducible $G$-modules are orthogonal, it is enough to consider $L_G(\lambda)$ with $\lambda \in X(T)^+$ a $2$-restricted dominant weight (Remark \ref{restrictedremark}). For groups of exceptional type, this leaves finitely many $\lambda$ to consider. For groups of classical type, we can further reduce the question to $G$ of type $A_l$ and type $C_l$. This follows from the next two lemmas. Note that in Lemma \ref{f:reductionB}, we identify the fundamental dominant weights of $B_l$ and $C_l$ by abuse of notation.

%

\begin{lemma}\label{f:reductionB}
Let $\lambda =  \sum_{i = 1}^{l} a_i \omega_i$, where $l \geq 2$ and $a_i \in \{0,1\}$ for all $1 \leq i \leq l$. 

\begin{enumerate}[\normalfont (i)]
\item Let $G$ be of type $B_l$ or $C_l$. If $a_l = 1$, then $V = L_G(\lambda)$ is orthogonal, except when $l = 2$ and $\lambda = \omega_2$.
\item The irreducible $B_l$-representation $L_{B_l}(\lambda)$ is orthogonal if and only if the irreducible $C_l$-representation $L_{C_l}(\lambda)$ is orthogonal.
\end{enumerate}
\end{lemma}

\begin{proof}

\begin{enumerate}[(i)]
\item Let $\varphi: B_l \rightarrow C_l$ be the usual exceptional isogeny between simply connected groups of type $B_l$ and $C_l$ \cite[Theorem 28]{SteinbergNotes}. Then $$L_{C_l}(\lambda)^\varphi \cong L_{B_l}( \sum_{i = 1}^{l-1} a_i \omega_i + 2a_l\omega_l) \cong L_{B_l}(\sum_{i = 1}^{l-1} a_i \omega_i) \otimes L_{B_l}(a_l\omega_l)^F$$ where the last equality follows by Steinberg's tensor product theorem. 

Assume that $a_l = 1$. Note that a $C_l$-module $V$ is orthogonal if and only if $V^\varphi$ is an orthogonal $B_l$-module. Thus it follows from Lemma \ref{tensorlemma} that $L_{C_l}(\lambda)$ is orthogonal, except possibly when $\lambda = \omega_l$. Finally, we know that $L_{C_l}(\omega_l)$ is orthogonal if and only if $l \geq 3$ by Example \ref{spinexample}. This proves the claim for $G$ of type $C_l$. 

For $G$ of type $B_l$, let $\tau: C_l \rightarrow B_l$ be the usual exceptional isogeny between simply connected groups of type $C_l$ and $B_l$. Then $$L_{B_l}(\lambda)^\tau \cong L_{C_l}( \sum_{i = 1}^{l-1} 2a_i \omega_i + a_l\omega_l) \cong L_{C_l}(\sum_{i = 1}^{l-1} a_i \omega_i)^F \otimes L_{C_l}(a_l\omega_l),$$ and now the claim follows as in the type $C_l$ case.

\item If $a_l = 1$, the claim follows from (i).  If $a_l = 0$, then $L_{C_l}(\lambda)^\varphi \cong L_{B_l}(\lambda)$ and the claim follows since a $C_l$-module $V$ is orthogonal if and only if $V^\varphi$ is an orthogonal $B_l$-module.
\end{enumerate}\end{proof}

\begin{lemma}\label{f:reductionD}
Let $\lambda =  \sum_{i = 1}^{l} a_i \omega_i$, where $l \geq 4$ and $a_i \in \{0, 1\}$ for all $1 \leq i \leq l$.

\begin{enumerate}[\normalfont (i)]
\item If $a_{l-1} \neq a_l$, then $L_{D_l}(\lambda)$ is orthogonal if it is self-dual.
\item If $a_{l-1} = a_l$, then $L_{D_l}(\lambda)$ is orthogonal if and only if $L_{C_l}(\sum_{i = 1}^{l-1} a_i \omega_i)$ is orthogonal, except when $\lambda = \omega_1$.
\end{enumerate}
\end{lemma}

\begin{proof}
\begin{enumerate}[(i)]
\item If $a_{l-1} \neq a_l$, then for example by \cite[Table 13.1]{Humphreys} we see that for $G$ of type $D_l$, the weight $\lambda$ is not a sum of roots. Therefore $L_G(0)$ cannot be a composition factor of $V_G(\lambda)$, and thus $L_G(\lambda)$ is orthogonal by Proposition \ref{char2quad} (iv).

\item Suppose that $a_{l-1} = a_l$. Considering $D_l < C_l$ as the subsystem subgroup generated by long roots, we have $L_{C_l}(\sum_{i = 1}^{l-1} a_i \omega_i) \downarrow D_l \cong L_{D_l}(\lambda)$ by \cite[Theorem 4.1]{SeitzClassical}. Now the claim follows from Lemma \ref{lemma:typeDrestriction}.
\end{enumerate}
\end{proof}



\subsection{Application to maximal subgroups of classical groups}\label{subsec:applmax}

\noindent In this subsection only, we allow $\Ch \fieldsymbol$ to be arbitrary.

As mentioned in the introduction, one motivation for Problem \ref{mainproblem} is in the study of maximal closed connected subgroups of classical groups. Let $\operatorname{Cl}(V)$ be a classical simple algebraic group, that is, $\operatorname{Cl}(V) = \SL(V)$, $\operatorname{Cl}(V) = \Sp(V)$, or $\operatorname{Cl}(V) = \SO(V)$. Finding maximal closed connected subgroups of $\operatorname{Cl}(V)$ can be reduced to the representation theory of simple algebraic groups. We proceed to explain how this is done. For more details, see \cite{LiebeckSeitz} and \cite{SeitzClassical}.

In \cite{LiebeckSeitz}, certain collections $\mathscr{C}_1$, $\ldots$, $\mathscr{C}_6$ of \emph{geometric subgroups} were defined in terms of the natural module $V$ and its geometry. A reduction theorem due to Liebeck and Seitz \cite[Theorem 1]{LiebeckSeitz} implies that for a positive-dimensional maximal closed subgroup $X$ of $\operatorname{Cl}(V)$ one of the following holds:

\begin{enumerate}[(i)]
\item $X$ belongs to some $\mathscr{C}_i$,
\item The connected component $X^\circ$ is simple, and $V \downarrow X^\circ$ is irreducible and tensor-indecomposable.
\end{enumerate}

In particular, the reduction theorem implies the following.

\begin{lause}\label{maximalreduction}
Let $X < \operatorname{Cl}(V)$ be a subgroup maximal among the closed connected subgroups of $\operatorname{Cl}(V)$. Then one of the following holds:

\begin{enumerate}[\normalfont (i)]
\item $X$ is contained in a member of some $\mathscr{C}_i$,
\item $X$ is simple, and $V \downarrow X$ is irreducible and tensor-indecomposable.
\end{enumerate}

\end{lause}


The maximal closed connected subgroups in case (i) of Theorem \ref{maximalreduction} are well understood \cite[Theorem 3]{SeitzClassical}. Furthermore, the maximal closed connected subgroups occurring in case (ii) of Theorem \ref{maximalreduction} can also be described. These were essentially determined by Seitz \cite{SeitzClassical} and Testerman \cite{TestermanIrr}. The result can be stated in the following theorem, which tells when an irreducible tensor-indecomposable subgroup is not maximal.

\begin{lause}\label{SeitzRestriction}
Let $Y$ be a simple algebraic group and let $V$ be a non-trivial irreducible tensor-indecomposable $p$-restricted and rational $Y$-module. If $X$ is a closed proper connected subgroup of $Y$ such that $X$ is simple and $V \downarrow X$ is irreducible, then $(X, Y, V)$ occurs in \cite[Table 1]{SeitzClassical}.
\end{lause}

To refine the characterization of maximal closed connected subgroups of $\operatorname{Cl}(V)$ given in \cite[Theorem 3]{SeitzClassical}, one should determine which of $\SL(V)$, $\Sp(V)$ and $\SO(V)$ contain $X$ and $Y$ in Theorem \ref{SeitzRestriction}. 

For example, let $Y$ be simple of type $D_5$ and let $X < Y$ be simple of type $B_4$ embedded in the usual way. Then for $V = L_Y(\omega_5)$ we have $V \downarrow X = L_X(\omega_4)$; this situation corresponds to entry $\operatorname{IV}_1$ in \cite[Table 1]{SeitzClassical}. Here $V$ is not self-dual as a $Y$-module, so $Y < \SL(V)$ only. However, $V \downarrow X$ is self-dual and $X < \SO(V)$ if $p \neq 2$, and $X < \SO(V) < \Sp(V)$ if $p = 2$ (see Table \ref{dualitytable} and Theorem \ref{mainprop}). In this situation $Y$ is maximal in $\SL(V)$, while $X$ is maximal in $\SO(V)$.

In fact, the results we have presented in this text allow one to determine for almost all $(X, Y, V)$ occurring in \cite[Table 1]{SeitzClassical} whether $V \downarrow X$ and $V \downarrow Y$ are orthogonal, symplectic, both, or neither. If $p \neq 2$, then this is easily done using Table \ref{dualitytable}. 

For $p = 2$, we list this information in Table \ref{table:embeddings}, which is deduced as follows. Entry $\operatorname{IV}_1$ is a consequence of Lemma \ref{f:reductionB}, Example \ref{spinexample} and Lemma \ref{f:reductionD}. In entry $S_3$, we have $V \downarrow Y = L_{C_3}(\omega_2)$ which is orthogonal by Example \ref{exampleC2}, and thus $V \downarrow X$ is also orthogonal. In entry $S_4$, we have $V \downarrow Y = L_{C_3}(\omega_1 + \omega_2)$, which is orthogonal by Proposition \ref{char2quad} (iv) since $V_{C_3}(\omega_1 + \omega_2)$ is irreducible. Entry $\operatorname{S}_6$ follows from Lemma \ref{tensorlemma} and Lemma \ref{f:reductionD}. In entries $\operatorname{S}_7$, $\operatorname{S}_8$, and $\operatorname{S}_9$, we have $V \downarrow X = L_X(\lambda)$ and $V_X(\lambda)$ is irreducible, so $V \downarrow X$ is orthogonal by Proposition \ref{char2quad} (iv). Entries $\operatorname{MR}_2$ and $\operatorname{MR}_3$ follow from Proposition \ref{F4result}, which show that $V \downarrow Y$ is orthogonal. Entry $\operatorname{MR}_5$ is a consequence of Lemma \ref{f:reductionB} (i).


\begin{table}[htp]
\centering
\begin{tabular}{| l | l | l | l |}
\hline
No. & $X < Y$ & $V \downarrow X$ & $V \downarrow Y$ \\
\hline

$\operatorname{IV}_1$ & $B_l < D_{l+1}$ & orthogonal & \begin{tabular}{@{}l@{}} $l+1$ even: orthogonal \\  $l+1$ odd: not self-dual \end{tabular}  \\
& & & \\
$\operatorname{S}_3$  & $G_2 < C_3$ & orthogonal & orthogonal  \\
$\operatorname{S}_4$  & $G_2 < C_3$ & orthogonal & orthogonal  \\
& & & \\
$\operatorname{S}_6$  & $B_{n_1} \cdots B_{n_k} < D_{1 + \sum n_i}$ & orthogonal & \begin{tabular}{@{}l@{}} $1 + \sum n_i$ even: orthogonal \\  $1 + \sum n_i$ odd: not self-dual \end{tabular} \\
& & & \\
$\operatorname{S}_7$  & $A_3 < D_7$ & orthogonal & not self-dual\\
$\operatorname{S}_8$  & $D_4 < D_{13}$ & orthogonal & not self-dual \\
$\operatorname{S}_9$  & $C_4 < D_{13}$ & orthogonal & not self-dual  \\
$\operatorname{MR}_2$ & $D_4 < F_4$ & orthogonal & orthogonal \\
$\operatorname{MR}_3$ & $C_4 < F_4$ & orthogonal & orthogonal \\
$\operatorname{MR}_4$ & $D_l < C_l$ & ? & ? \\
$\operatorname{MR}_5$ & $B_{n_1} \cdots B_{n_k} < B_{n_1 + \cdots + n_k}$ & orthogonal & orthogonal \\
\hline
\end{tabular}
\caption{Invariant forms on $V \downarrow X$ and $V \downarrow Y$ for $(X, Y, V)$ occurring in \cite[Table 1]{SeitzClassical} in the case $p = 2$.}\label{table:embeddings}
\end{table}

What remains is the entry $\operatorname{MR}_4$ from \cite[Table 1]{SeitzClassical}. Here $X = D_l$ ($l \geq 4$) embedded in $Y = C_l$ as the subsystem subgroup of long roots, and we have $V = L_Y(\sum_{i = 1}^{l-1} a_i \omega_i)$ with $a_i \in \{0,1\}$, and $V \downarrow X = L_X(\sum_{i = 1}^{l-2} a_i \omega_i + a_{l-1}(\omega_{l-1} + \omega_l))$. In this situation we do not know in general whether $V \downarrow Y$ and $V \downarrow X$ are orthogonal, but we do know that except in the case where $\sum_{i = 1}^{l-1} a_i \omega_i = \omega_1$, it is true that $V \downarrow Y$ is orthogonal if and only if $V \downarrow X$ is orthogonal (Lemma \ref{f:reductionD}). Using this fact and the information in Table \ref{table:embeddings}, we can deduce the following result.

\begin{lause}\label{theorem:maxconseq}
Let $Y$ be a simple algebraic group and let $V$ be a non-trivial irreducible tensor-indecomposable $p$-restricted $Y$-module. If $X$ is a closed proper connected subgroup of $Y$ such that $X$ is simple and $V \downarrow X$ is irreducible, then one of the following holds.

\begin{enumerate}[\normalfont (i)]
\item $V \downarrow Y$ is not self-dual.
\item Both $V \downarrow Y$ and $V \downarrow X$ are orthogonal.
\item Neither of $V \downarrow Y$ or $V \downarrow X$ is orthogonal.
\item $p = 2$, $X$ is of type $D_l$, $Y$ is of type $C_l$ and $V$ is the natural module of $Y$.
\end{enumerate}
\end{lause}

\subsection{Fundamental self-dual irreducible representations}

\noindent Among the irreducible self-dual $G$-modules that are not orthogonal, so far the only ones that we know of are in some sense minimal among the self-dual irreducible modules of $G$. We make this more precise in what follows, and pose the question whether any other examples can be found.

Recall that $L_G(\lambda)$ is self-dual if and only if $\lambda = -w_0(\lambda)$, where $w_0$ is the longest element in the Weyl group. We know that any dominant weight $\lambda \in X(T)^+$ can be written uniquely as a sum of fundamental dominant weights, that is, $\lambda = \sum_{i = 1}^l a_i \omega_i$ for unique integers $a_i \geq 0$. Now similarly, there exists a collection $\mu_1, \ldots, \mu_t \in X(T)^+$ such that $\mu_i = -w_0(\mu_i)$ for all $i$, and such that any $\lambda \in X(T)^+$ with $\lambda = -w_0(\lambda)$ can be written uniquely as $\sum_{i = 1}^t a_i \mu_i$ with $a_i \geq 0$. For each simple type, these $\mu_i$ are listed below.

\begin{itemize}
\item Type $A_l$ ($l$ odd): $\mu_i = \omega_i + \omega_{l+1-i}$ for $1 \leq i \leq \frac{l-1}{2}$, and $\mu_{\frac{l+1}{2}} = \omega_{\frac{l+1}{2}}$.
\item Type $A_l$ ($l$ even): $\mu_i = \omega_i + \omega_{l+1-i}$ for $1 \leq i \leq \frac{l}{2}$.
\item Types $B_l$, $C_l$, $D_l$ ($l$ even), $G_2$, $F_4$, $E_7$, and $E_8$: $\mu_i = \omega_i$ for $1 \leq i \leq \rank G$.
\item Type $D_l$ ($l$ odd): $\mu_i = \omega_i$ for $1 \leq i \leq l-2$, and $\mu_{l-1} = \omega_{l-1} + \omega_l$.
\item Type $E_6$: $\mu_1 = \omega_1 + \omega_6$, $\mu_2 = \omega_2$, $\mu_3 = \omega_3 + \omega_5$, and $\mu_4 = \omega_4$. 
\end{itemize}

Currently the only known examples of non-trivial irreducible modules $L_G(\lambda)$ that are self-dual and not orthogonal are of the form $L_G(\mu_i)$. Are there any others?

\begin{prob}\label{tensorproblem}
Let $\lambda \in X(T)^+$ be $2$-restricted and suppose that $\lambda = \lambda_1 + \lambda_2$, where $\lambda_i \in X(T)^+$ are nonzero and $-w_0(\lambda_i) = \lambda_i$. Is $L_G(\lambda)$ orthogonal? 
\end{prob}

If the answer to Problem \ref{tensorproblem} is yes, then our results would settle Problem \ref{mainproblem} almost completely. Indeed, a positive answer to Problem \ref{tensorproblem} would show that any non-orthogonal self-dual irreducible representation of $G$ must be equal to a Frobenius twist of $L_G(\mu_i)$ for some $i$. Our results determine the orthogonality of $L_G(\mu_i)$ when $G$ is of classical type. The non-orthogonal ones for type $A_l$ are the $L_{A_l}(\omega_i + \omega_{l+1-i})$ described in Theorem \ref{quadpropA}. For $G$ of type $B_l$, $C_l$, or $D_l$, the non-orthogonal ones are $L_G(\omega_i)$ as described in Theorem \ref{mainprop}, with the unique exception of $L_G(\omega_4 + \omega_5)$ for $G$ of type $D_5$ (arising from restriction of $L_{C_5}(\omega_4)$ to $G$).

Then a handful of $\mu_i$ still remain for exceptional types. For $G$ simple of exceptional type, the irreducibles $L_G(\mu_i)$ whose orthogonality was not decided in Section \ref{exceptionalsection} are as follows.

\begin{itemize}
\item $L_G(\omega_3 + \omega_5)$ for $G$ of type $E_6$,
\item $L_G(\omega_3)$ and $L_G(\omega_4)$ for $G$ of type $E_7$,
\item $L_G(\omega_i)$ for $2 \leq i \leq 6$ for $G$ of type $E_8$.
\end{itemize}

In any case, a natural next step towards solving Problem \ref{mainproblem} should be determining an answer to Problem \ref{tensorproblem}. The methods we have used in this paper to solve Problem \ref{mainproblem} for certain families of $L_G(\lambda)$ rely heavily on detailed information about the structure of the Weyl module $V_G(\lambda)$, which is not known in general. For small-dimensional representations the composition factors of $V_G(\lambda)$ can be found using the results of L{\"u}beck given in \cite{Lubeck} and \cite{LubeckWebsite}. However, in general this sort of information is not available, and in characteristic $2$ the composition factors of $V_G(\lambda)$ are known only in a relatively few cases. For example, for $G$ of type $E_8$ we do not even know the dimension of $L_G(\omega_i)$ for all $i$ in characteristic $2$.

\subsection{Fixed point spaces of unipotent elements}

\noindent We finish by a question about a possible orthogonality criterion for irreducible representations. Let $\varphi: G \rightarrow \SL(V)$ be a non-trivial irreducible representation of $G$. Assume that $V$ is self-dual, so that $\varphi(G) < \Sp(V)$ (Lemma \ref{fonglemma}). If $V$ is an orthogonal $G$-module, then $\varphi(G) < \operatorname{O}(V)$ and so $\varphi(G) < \SO(V)$ since $G$ is connected. Then for any unipotent element $u \in G$, the number of Jordan blocks of $\varphi(u)$ is even \cite[Proposition 6.22]{LiebeckSeitzClass}. In other words, for all $u \in G$ we have that $\dim V^u$ is even, where $V^u$ is the subspace of fixed points for $u$. Does the converse hold?

\begin{prob}\label{prob:unip}
Let $V$ be a non-trivial irreducible self-dual representation of $G$. If $V$ is not orthogonal, does there exist a unipotent element $u \in G$ such that $\dim V^u$ is odd?
\end{prob}

In Table \ref{table:unipdim}, we have listed examples (without proof) of some non-orthogonal representations $V$ of $G$ for which the answer to Problem \ref{prob:unip} is yes. If the answer to Problem \ref{prob:unip} turns out to be yes, we would have an interesting criterion for an irreducible representation $V$ of $G$ to be orthogonal. A positive answer would show that the orthogonality of an irreducible representation can be decided from the properties of individual elements of $G$.

\begin{table}[htp]
\centering
\begin{tabular}{| l | l | l | l |}
\hline
Type of $G$ & $V$ & Conjugacy class of $u$ & $\dim V^u$ \\
\hline
$A_{l}$, $l + 1 \equiv 2 \mod{4}$ & $L_G(\omega_1 + \omega_l)$ & regular & $2l+1$ \\
$C_l$ & $L_G(\omega_1)$ & regular & $1$ \\
$C_l$, $l \equiv 2 \mod{4}$ & $L_G(\omega_2)$ & regular & $l-1$ \\
$A_4$ & $L_G(\omega_2 + \omega_3)$ & $A_3$ & $19$ \\
$A_5$ & $L_G(\omega_2 + \omega_4)$ & regular & $21$ \\
$C_5$ & $L_G(\omega_4)$ & regular in $D_5$ & $21$ \\
$C_6$ & $L_G(\omega_4)$ & regular & $25$ \\
$G_2$ & $L_G(\omega_1)$ & regular & $1$ \\
$E_7$ & $L_G(\omega_1)$ & regular & $7$ \\ 
\hline
\end{tabular}
\caption{Non-orthogonal irreducible representations $V$ of $G$ with examples of $\dim V^u$ odd for some unipotent element $u \in G$.}\label{table:unipdim}
\end{table}

\section*{References}
\bibliographystyle{alpha}
\bibliography{bibliography}

\end{document}